\newif\ifplainarticle
\plainarticletrue

\IfFileExists{maxwell}
{\plainarticletrue}
{}

\ifplainarticle
\documentclass{article}

\IfFileExists{MinionPro.sty}{\usepackage[lf]{MinionPro}}{}
\usepackage{amsmath,amsthm,amsbsy}
\IfFileExists{MinionPro.sty}{}{\usepackage{times,txfonts}}

\usepackage{geometry}
\usepackage{color}
\definecolor{labelkey}{rgb}{0,0,.75}
\definecolor{MyGreen}{rgb}{0,.6,.2}
\definecolor{MyDarkBlue}{rgb}{.1,.1,.75}
\usepackage[colorlinks=true, citecolor=MyGreen, linkcolor=MyDarkBlue]{hyperref}
\usepackage{tensor}
\usepackage[all,cmtip]{xy}
\usepackage{graphicx}
\usepackage{enumitem}

\numberwithin{equation}{section}

\else

\documentclass[12pt]{amsart}
\usepackage{a4}
\usepackage{graphicx}
\usepackage{amsmath,amssymb,amsthm, color, fancyvrb, array}
\usepackage{outlines}
\usepackage{enumitem}

\usepackage{times,txfonts}
\usepackage{geometry}
\fi

\theoremstyle{plain}
\newtheorem{thm}{Theorem}[section]
\newtheorem{lem}[thm]{Lemma}
\newtheorem{cor}[thm]{Corollary}
\newtheorem{prop}[thm]{Proposition}
\theoremstyle{definition}
\newtheorem{defn}[thm]{Definition}

\newcommand{\R}{{\mathbb R}}

\newcommand{\tr}{\mathrm{tr}}
\newcommand{\di}{\mathrm{div}}

\newcommand{\p}{\partial}
\newcommand{\Vol}{\textrm{Vol}}
\newcommand{\calL}{\mathcal{L}}
\let\Reals\R
\DeclareMathOperator{\Lap}{\Delta}
\newcommand{\loc}{\text{loc}}

\newcommand{\Mbar}{\overline{M}}
\newcommand{\gbar}{{\overline{g}}}
\newcommand{\kbar}{\overline{k}}

\newcommand{\pbar}{\bar \partial}
\newcommand{\phibar}{\overline{\phi}}

\newcommand{\Vbar}{\overline{V}}

\newcommand{\ra}{\rightarrow}

\newcounter{mnotecount}[section]
\let\oldmarginpar\marginpar
\renewcommand\marginpar[1]{\-\oldmarginpar[\raggedleft\footnotesize #1]%
{\raggedright\footnotesize #1}}

\ifplainarticle

\title{Yamabe Classification and Prescribed Scalar Curvature 
in the Asymptotically {E}uclidean Setting}

\author{James Dilts and David Maxwell}

\else
\title[Yamabe and Scalar Curvature Problems on AE Manifolds]{The Yamabe Classification and
Prescribed Scalar Curvature Problems on Asymptotically Euclidean Manifolds}

\author[J. Dilts]{James Dilts$^1$}
\address{$^1$University of Oregon}
\email{jdilts@uoregon.edu}

\author[D. Maxwell]{David Maxwell$^2$}
\address{$^2$University of Alaska, Fairbanks}
\email{damaxwell@alaska.edu}


\fi

\date{March 13, 2015}

\begin{document}
\maketitle

\begin{abstract}
We prove a necessary and sufficient condition for an asymptotically Euclidean
manifold to be conformally related to one with specified nonpositive scalar curvature:
the zero set of the desired scalar curvature must have a positive Yamabe invariant, as
defined in the article.  We show additionally how the sign of the Yamabe invariant 
of a measurable set can be computed from the sign of certain generalized ``weighted" 
eigenvalues of the conformal Laplacian. Using the prescribed scalar curvature
result we give a characterization of the Yamabe classes of asymptotically Euclidean
manifolds. We also show that the Yamabe class of an asymptotically Euclidean
manifold is the same as the Yamabe class of its conformal compactification.
\end{abstract}

\section{Introduction}
One formulation of the prescribed scalar curvature problem asks, for a given Riemannian
manifold
$(M^n,g)$ and some function $R'$, is there a conformally related metric $g'$ with scalar
curvature $R'$? If we define $g' = \phi^{N-2} g$ for $N:= \frac{2n}{n-2}$, this is 
equivalent to finding a positive solution of
\begin{equation}\label{eq:prescribed-sc}
  -a\Delta \phi + R \phi = R' \phi^{N-1},
\end{equation} where $a := \frac{4(n-1)}{n-2}$, $R$ is the scalar curvature of $g$,
and $-a\Delta + R$ is the conformal Laplacian. 

On a compact manifold the Yamabe invariant of the conformal
class of $g$ poses an obstacle to the solution of \eqref{eq:prescribed-sc}.
For example, in the case where $M$ is connected and $R'$ is constant, problem \eqref{eq:prescribed-sc}
is known as the Yamabe problem, and it admits a solution if and only
if the sign of the Yamabe invariant agrees with the sign of $R'$ 
\cite{Yamabe60}\cite{Trudinger68}\cite{Aubin76}\cite{Schoen84}.
More generally,
if $R'$ has constant sign, we can conformally transform to a metric with
scalar curvature $R'$ only if the sign of the Yamabe invariant agrees with the
sign of the scalar curvature.  Hence it is natural to divide conformal classes
into three types, Yamabe positive, negative, and null, depending on the sign
of the Yamabe invariant.

We are interested in 
solving equation \eqref{eq:prescribed-sc} on a
class of complete Riemannian manifolds that, loosely speaking, have a geometry
approximating Euclidean space at infinity.  These asymptotically Euclidean
(AE) manifolds also possess a Yamabe invariant, but the relationship between
the Yamabe invariant and problem \eqref{eq:prescribed-sc} is 
not well understood in the AE setting, except for some results
concerning Yamabe positive metrics.
We have the following consequences of \cite{Maxwell05b} Proposition 3.
\begin{enumerate}
\item An AE metric can be conformally transformed to an AE
metric with zero scalar curvature if and only if it is Yamabe positive.
As a consequence, since the scalar curvature of an AE metric decays to zero at infinity,
only Yamabe positive AE metrics can be conformally transformed to 
have constant scalar curvature.
\item Yamabe positive AE metrics have conformally related AE metrics with
everywhere positive scalar curvature, and 
conformally related AE metrics with everywhere negative scalar curvature.
\item If an AE metric admits a conformally
related metric with non-negative scalar curvature, then it is Yamabe positive.
\end{enumerate}
Note that it was originally believed that transformation to zero scalar
curvature is possible if and only if the manifold is Yamabe non-negative
\cite{BC81}.
The proof in \cite{BC81} contains an error, and the statement and proof
were corrected in \cite{Maxwell05b}. See also \cite{Friedrich11}, which
shows that there exists a Yamabe-null AE manifold and hence the 
hypotheses of \cite{BC81} and \cite{Maxwell05b} are genuinely different.

As a consequence of these three facts, the situation on an AE manifold is somewhat different from the
compact setting.  In particular, 
although positive scalar curvature is a hallmark
of Yamabe positive metrics, negative scalar curvature does
not characterize Yamabe-negative metrics.  Indeed, we show in this article
that given an AE metric $g$, and a strictly negative
function $R'$ that decays to zero suitably at infinity, the conformal
class of $g$ includes a metric with scalar curvature $R'$ regardless
of the sign of the Yamabe invariant.  So every strictly negative scalar curvature 
is attainable for every conformal class, but zero
scalar curvature is attainable only for Yamabe positive metrics.  Thus we
are lead to investigate the role of the Yamabe class
in the boundary case of prescribed non-positive scalar curvature.

Rauzy treated the analogous problem on
smooth compact Riemannian manifolds in \cite{Rauzy95},
which contains the following statement.
Suppose $R'\le 0$ and $R'\not\equiv 0$.  Observe that if
$R'$ is the scalar curvature of a metric conformally related to some $g$,
then $g$ must be Yamabe-negative, and without
loss of generality we assume that $g$ has constant negative scalar curvature $R$.
Then there is a metric in the conformal
class of $g$ with scalar curvature $R'$ if and only if
\begin{equation}\label{eq:rauzycond}
a \lambda_{R'} > -R
\end{equation}
where $a$ is the constant from equation \eqref{eq:prescribed-sc} and where
\begin{equation}
\lambda_{R'} = \inf \left\{ \frac{\int |\nabla u|^2 }{\int u^2}: u\in W^{1,2}, u\ge 0, u\not \equiv 0, \int R' u= 0  \right\}.
\end{equation}
Rauzy's condition \eqref{eq:rauzycond}
is not immediately applicable on asymptotically Euclidean manifolds, in part
because of the initial transformation to constant negative scalar curvature.
However, recalling that $R$ is constant 
we can write $a\lambda_{R'}+R$ as the infimum of
\begin{equation}
\frac{\int a|\nabla u|^2 + R u^2 }{\int u^2}
\end{equation}
over functions $u$ supported in the region where $R'=0$.  So, morally, inequality
\eqref{eq:rauzycond} expresses the positivity of the first eigenvalue of 
the conformal Laplacian of the constant scalar curvature metric $g$ 
on the region $\{R'=0\}$.  The connection between the first eigenvalue
of the conformal Laplacian and prescribed scalar curvature problems
is well known, but its use is more technical on
non-compact manifolds where true eigenfunctions need not exist. For example,
\cite{FCS80} shows that a metric on a noncompact manifold
can be conformally transformed to a scalar flat one if and only if the first eigenvalue
of the conformal Laplacian is positive on every bounded domain.

In this article we extend these ideas in a number of ways 
to solve the prescribed non-positive scalar curvature problem on 
asymptotically Euclidean manifolds, and we obtain
a related characterization of the Yamabe class of
an AE metric. In particular, we show the following.
\begin{itemize}
\item Every measurable subset $V\subseteq M$ 
can be assigned a number $y(V)$ that generalizes 
the Yamabe invariant of a manifold.  The invariant
depends on the conformal class of the AE metric,
but is independent of the conformal representative.
\item We can assign every measurable subset $V\subseteq M$
a number $\lambda_{\delta}(V)$ that generalizes the
first eigenvalue of the conformal Laplacian.  These
numbers are not conformal invariants, and are not
even canonically defined as they depend on a choice
of parameters (a number $\delta$ and a choice
of weight function at infinity).  Nevertheless
the sign of $\lambda_\delta(V)$ agrees with the sign
of $y(V)$, regardless of the choice of these parameters.
\item Given a candidate scalar curvature $R'\le 0$,
there is a metric in the conformal class of $g$
with scalar curvature $R'$ if and only if 
$\{R'=0\}$ is Yamabe positive, i.e., $y(\{R'=0\})>0$.
\item A metric is Yamabe positive if and only if
every scalar curvature $R'\le 0$ is attained by
a metric in its conformal class.
\item A metric is Yamabe null if and only if
every scalar curvature $R'\le 0$, except for $R'\equiv 0$,
is attained by a metric in its conformal class.
\item A metric is Yamabe negative if and only if
there is a scalar curvature $R'\le 0$, $R'\not\equiv 0$,
that is unattainable within the conformal class.  We also
present some results concerning which scalar curvatures
have Yamabe positive zero sets, and hence are attainable.
\item Additionally, a metric is Yamabe positive/negative/null
if and only if it admits a conformal compactification
to a metric with the same Yamabe type.
\end{itemize}

These results carry over to compact manifolds, where
we obtain some technical improvements.  First, 
Rauzy's condition \eqref{eq:rauzycond} is
equivalent to our condition $y(\{R'=0\})>0$  
(or equivalently $\lambda_\delta(\{R'=0\})>0$).  But the condition
$y(\{R'<0\})>0$
can be measured without reference to a particular background metric.
Moreover, we work with fairly general metrics ($W^{2,p}_\loc$ with $p>n/2$),
and arbitrary scalar curvatures in $L^p(M)$. 
Finally, there is an error in Rauzy's proof, closely
related to the gap in Yamabe's original attempt at
the Yamabe problem, that we correct in our presentation.
\footnote{We would like to thank Rafe Mazzeo for having spotted
our own error in this regard while this work was in preparation.}

The prescribed scalar curvature problem on AE manifolds when $R'\ge0$, or when
$R'$ changes sign, remains open.  
Of course if $R'\ge 0$ the problem
can only be solved if the manifold is Yamabe positive, but it is not
known the extent to which this is sufficient.  For scalar curvatures
that change sign, little is known for any Yamabe class.  Nevertheless,
the case $R'\le 0$ that we treat here has an interesting application to general relativity; see below.
For comparison, we note that the prescribed scalar curvature problem on a compact manifold
is also not yet fully solved.  On a Yamabe-positive manifold it is necessary
that $R'>0$ somewhere, and on a Yamabe-null manifold it is necessary
that either $R'\equiv 0$, or $R'>0$ somewhere and $\int R'<0$ when
computed with respect to the scalar flat conformal representative.  See
\cite{ES86}
which shows that these conditions
are sufficient in some cases.  See also \cite{Bourguignon:1987ge}
for obstructions posed by conformal Killing fields.

Our interest in this problem stems from its application to 
general relativity. Initial data for the Cauchy problem
must satisfy certain compatibility conditions known as the 
Einstein constraint equations.  One approach to finding solutions
of the constraint equations, the so-called conformal method,
involves solving a coupled system of PDEs that includes 
the Lichnerowicz equation, which in the vacuum case is
\begin{equation}\label{eq:lichintro}
-a\Delta \phi  +R \phi + \frac{n-1}{n} \tau^2 \phi^{N-1} - \beta^2 \phi^{-N-1} = 0.
\end{equation} 
Here $\phi$ is an unknown conformal factor, $\tau$ is a prescribed function
(a mean curvature, in fact),
and, for the discussion at hand, $\beta$ can be thought of as a prescribed function
as well.  On a compact Yamabe-negative manifold, the Lichnerowicz
equation \eqref{eq:lichintro} is solvable if and only if the prescribed
scalar curvature problem \eqref{eq:prescribed-sc} is solvable for $R'=-\tau^2$
\cite{Maxwell05}.  An analogous condition holds on AE manifolds \cite{DGI13},
and hence the prescribed non-negative scalar curvature problem is
intimately connected to the solvability of the Lichnerowicz equation.
In particular, on an AE or Yamabe-negative compact manifold,
the Lichnerowicz equation can only be solved if the zero set 
of the mean curvature $\tau$ is Yamabe positive.

\section{Asymptotically Euclidean Manifolds}\label{sec:AEManifolds}

Throughout this article we assume that $(M,g)$ is a connected Riemannian
$n$-manifold.  
An asymptotically Euclidean (AE) manifold is a complete manifold such that
for some compact $K \subset M$, the complement $M\setminus K$ has finitely many
components $E_i$, with each $E_i$ admitting a distinguished diffeomorphism 
to the exterior of a ball in $\R^n$. The $E_i$ are called the ends of $M$, 
and in end coordinates the metric $g$ decays at infinity to 
the standard Euclidean metric $e$.

In order to make this notion precise we use weighted Sobolev spaces. Let
$\rho\geq 1$ be a smooth function on $M$ that agrees with
the Euclidean radial coordinate function near infinity on each end,
and let $\hat g$ be a smooth metric on $M$ that equals the Euclidean
metric in a neighborhood of each infinity. We say that a $L^1_{\loc}$ 
tensor $T$ belongs to $W^{k,p}_{\delta}(M)$ if
\begin{equation}\label{eq:weighted-norm}
\|T\|_{W^{k,p}_\delta(M)}:=\sum_{j=0}^k \left\|\rho^{-\delta - \frac{n}{p} + j} 
\; |\nabla^j T | \right\|_{L^p(M)} < \infty
\end{equation}
where all metric quantities in equation \eqref{eq:weighted-norm} 
use $\hat g$.  When $k=0$, we denote the space by $L^{p}_{\delta}(M)$ with
norm $\|\cdot\|_{p,\delta}$.  
It is easy to see that the spaces $W^{k,p}_\delta$ are independent of the choice
of background metric $\hat g$, and that the associated norms are equivalent.
There are varying conventions in the literature for
the weight parameter $\delta$ in equation \eqref{eq:weighted-norm},
and we follow \cite{Bartnik86}.  Consequently,
functions in $W^{k,p}_{\delta}$ have asymptotic growth
$O(r^\delta)$ on each end. Other properties of weighted spaces
can be found in \cite{Bartnik86},
and they parallel those
for Sobolev spaces on compact manifolds.  There are two key subtleties. 
First, $L^p_\delta$ embeds in $L^{p'}_{\delta'}$ 
if $p>p'$ and $\delta<\delta'$, but this is not true if $\delta=\delta'$.
Second, the embedding
\begin{equation}
W^{k,p}_{\delta} \hookrightarrow W^{k-1,p'}_{\delta'}
\end{equation}
is compact so long both
\begin{equation}
 \frac1p -\frac1n < \frac1{p'} \qquad\text{and}\qquad \delta<\delta'.
\end{equation}

We also have Sobolev embedding into spaces of continuous functions.
A function $u$ belongs
to $L^\infty_\delta(M)$ if
\begin{equation}
\|u\|_{L^\infty_\delta}(M) := \sup_M |u|\rho^{-\delta} <\infty
\end{equation} 
and we write $C^0_\delta$ for the continuous elements of $L^\infty_\delta$.
Then $W^{k,p}_\delta \subset C^0_\delta$
for $p>n/k$ \cite{Bartnik86}. 

We say that $g$ is a $W^{k,p}_\tau$ AE metric if $\tau <0$ and
\begin{equation}
g - \hat g \in W^{k,p}_\tau.
\end{equation}
We will work exclusively with $W^{2,p}_{\tau}$ AE metrics with $p>n/2$,
and we henceforth assume 
\begin{equation}
p>n/2 \quad\text{and}\quad \tau<0.
\end{equation}
A $W^{2,p}_\tau$
metric is H\"older continuous and has curvatures in $L^p_{\tau-2}$.
Using the fact that $W^{2,p}_{\loc}$ is an algebra, a straightforward
computation shows that we can use a $W^{2,p}_\tau$ metric 
for the metric quantities
in equation \eqref{eq:weighted-norm} to obtain an equivalent norm,
so long as $0\le k\le 2$. We will use this definition of the
norm whenever it is appropriate.  

The 
Laplacian and conformal Laplacian of a $W^{2,p}_\tau$  metric are well-defined 
as maps from $W^{2,q}_\delta$ to $L^q_{\delta-2}$ for $q\in (1,p]$, they are Fredholm
with index 0 if $\delta\in (2-n,0)$, and indeed the Laplacian is an isomorphism
in this range; see, e.g., \cite{Bartnik86} Proposition 2.2.  Note that 
\cite{Bartnik86} works on a manifold diffeomorphic to $R^n$, but the results
we cite from \cite{Bartnik86} extend 
to manifolds with general topology and any finite number of ends.

Many of the results in this article hold for both asymptotically Euclidean
and compact manifolds, and indeed we can often 
treat a $W^{2,p}$ metric on a compact manifold
as a $W^{2,p}_\tau$ metric on an asymptotically Euclidean manifold with zero ends,
in which case the weight function $\rho$ is irrelevant and could be set to 1
if desired.  For the sake of brevity, throughout Section 3 we interpret 
a compact manifold as an AE manifold with zero ends.
In the remaining sections there are differences between the two cases and
we assume that AE manifolds have at least one end.

The weight parameter
\begin{equation}\label{eq:deltastar}
\delta^* = \frac{2-n}{2}
\end{equation}
plays a prominent role in this paper, and it reflects the minimum decay
needed to ensure $\int|\nabla u|^2$ is finite.  At this decay
rate, $L^N_{\delta^*}=L^N$ and we have the following inequalities
that generalize the Poincar\'e and Sobolev inequalities on $\mathbf{R}^n$.

\begin{lem}\label{lem:poincare}
Let $(M,g)$ be a non-compact $W^{2,p}_\tau$ AE manifold.
There exists constants $c_1,c_2$ such that
\begin{equation}\label{eq:poincare}
\|\nabla u\|_2 \ge c_1 \|u\|_{2,\delta^*}
\end{equation}
\begin{equation}\label{eq:sobolev}
\|\nabla u\|_2 \ge c_2 \|u\|_{N}
\end{equation}
for all $u\in W^{1,2}_{\delta^*}(M)$, where $\delta^*$ is defined 
in equation \eqref{eq:deltastar} and where $N$ is the critical Sobolev
exponent $2n/(n-2)$.
\end{lem}
\begin{proof}
Suppose to the contrary that we can find a sequence $\{u_k\}$
of smooth functions with $\|u_k\|_{2,\delta^*}=1$ and $\|\nabla u_k\|_{2}\ra 0$.
It then follows that $\{u_k\}$ is bounded in $W^{1,2}_{\delta^*}(M)$ and therefore
a subsequence (which we reduce to) converges to a weak limit $u\in W^{1,2}_{\delta^*}(M)$.
Since $\nabla u_k\ra 0$ in $L^2$ we conclude that $u$ is constant, and since $\delta^*<0$
we conclude that $u=0$.   Moreover, $u_k\ra 0$ strongly in $L^2$ on compact sets.

Let $\eta$ be a cutoff function that equals 1 outside of some large ball and has support
contained in the ends of $M$.  Since $\nabla u_k\ra 0$ in $L^2(M)$ and since $u_k\ra0$ in
$L^2$ on compact sets we see that $\nabla (\eta u_k)\ra 0$ in $L^2(M)$.  Also,
since $u_k\ra0$ in $L^2$ on compact sets it follows that $(1-\eta) u_k \ra 0$ in 
$L^2_{\delta^*}$. Since $\|u_k\|_{2,\delta^*}=1$ and since 
$\|(1-\eta)u_k\|_{2,\delta^*}\ra 0$ it follows that $\|\eta u_k\|_{2,\delta^*}\ra 1$.

From the weighted Poincar\'e inequality \cite{Bartnik86}
Theorem 1.3(ii) we know that there is a constant $c$ such that for all $k$,
\begin{equation}
\|\eta u_k\|_{\overline g,2,\delta^*}\le c\|\nabla (\eta u_k)\|_{\overline g,2}
\end{equation}
where $\overline g$ is the Euclidean metric on the end.  But $g$ is comparable to
$\overline g$ on the end, so this same inequality holds for $g$ after suitably
modifying $c$. This is a contradiction.

The proof of inequality \eqref{eq:sobolev} is essentially the same as 
\eqref{eq:poincare}.
\end{proof}

Lemma \ref{lem:poincare} fails on compact manifolds due to the presence of 
the constants.  For our proofs that treat the compact and
non-compact case simultaneously it will be helpful to have 
a suitable inequality that works in both settings.
Observe that for any $\delta>0$ there exists $c_2$ such that
\begin{equation}
\label{eq:sobolev-cpct}
\|u\|_{2,\delta}+\|\nabla u\|_2 \ge c_2 \|u\|_{N}.
\end{equation}
This follows from the standard Sobolev inequality on compact manifolds and
follows trivially from inequality \eqref{eq:sobolev} on non-compact manifolds.

\section{The Yamabe Invariant of a Measurable Set} \label{sec:FirstEigenvalue}

Throughout this section, let $(M,g)$ be a $W^{2,p}_\tau$ AE manifold 
with $p>n/2$ and $\tau<0$, with the convention that a compact manifold is 
an AE manifold with zero ends. 
For $u\in C_c^\infty(M)$, $u\not\equiv 0$, the Yamabe quotient of $u$ is
\begin{equation}
Q^y_g(u) = \frac{\int a|\nabla u|^2 + R u^2}{\|u\|_{N}^2}
\end{equation}
and the Yamabe invariant of $g$ is the infimum of $Q^y_g$ taken over
$C_c^\infty(M)$. Here and in other notations we will drop the decoration
$g$ when the metric is understood. Our principal goal in this section is to 
define a similar conformal invariant for arbitrary 
measurable subsets of $M$ and to analyze its properties.

It will be convenient to work with a complete function space, and we claim
that the domain of $Q^y$ can be extended to 
$W^{1,2}_{\delta^*}\setminus\{0\}$ where $\delta^*$ is defined in equation
\eqref{eq:deltastar}.
To see this, first note from the embedding properties of 
weighted Sobolev spaces that $W^{1,2}_{\delta^*}$ embeds continuously in $L^N=L^N_{\delta^*}$
and that $u\mapsto \nabla u$ is continuous from $W^{1,2}_{\delta^*}$ to $L^2$;
indeed $\delta^*$ is the minimum decay needed to ensure these conditions.
To treat the scalar curvature term in $Q^y$, we have the following.
\begin{lem}\label{lem:RIntegralBound}
The map
\begin{equation}\label{eq:loworder}
u\mapsto \int R u^2
\end{equation}
is weakly continuous on $W^{1,2}_{\delta^*}$.  Moreover, for any $\delta>\delta^*$
and $\epsilon>0$, there is constant $C>0$ such that
\begin{equation}\label{eq:RuSquaredEstimate}
\left| \int R u^2 \right| \le \epsilon \|\nabla u\|_2^2 + C \|u\|_{2,\delta}^2.
\end{equation}
\end{lem}
\begin{proof}
Recall that $R\in L^p_{\tau-2}$ where $p>n/2$ and $\tau<0$.
So there is an $s\in (0,1)$ such that
\begin{equation}
\frac 1p= s\frac 2n.
\end{equation}
Set $\sigma=\delta^*-\tau/2$.  Since $s<1$ and $\sigma>\delta^*$,
$W^{1,2}_{\delta^*}$ embeds compactly in $W^{s,2}_{\sigma}$,
where the interpolation space $W^{s,2}_{\sigma}$
is described in 
\cite{triebel-weighted-one}\cite{triebel-weighted-two}.
Moreover, $W^{s,2}_{\sigma}$ embeds continuously in $L^q_\sigma$
where
\begin{equation}
\frac{1}{q} = \frac{1}{2} - \frac{s}{n} =  \frac{1}{2}\left( 1-\frac{1}{p}\right).
\end{equation}
Since
\begin{equation}
\frac{1}{p} + \frac{2}{q} = 1
\end{equation}
and since
\begin{equation}
\tau-2 + 2\sigma = 2\delta^* -2 = -n,
\end{equation}
H\"older's inequality implies the map \eqref{eq:loworder}
is continuous on $L^q_\sigma$, and from the previously mentioned compact
embedding the map \eqref{eq:loworder} is therefore weakly continuous on $W^{1,2}_{\delta^*}$.
Moreover, H\"older's inequality implies there is a constant $C$ such that
\begin{equation}\label{eq:loworder1}
\left| \int Ru^2\right| \le C \|u\|_{W^{s,2}_\sigma}^2.
\end{equation}
From interpolation \cite{triebel-weighted-two} we have
\begin{equation}\label{eq:loworder2}
\|u\|_{W^{s,2}_\sigma} \le C \|u\|_{W^{1,2}_{\delta^*}}^s \|u\|_{2,\delta}^{1-s}
\end{equation}
where $\delta$ satisfies
\begin{equation}
s\delta^* + (1-s)\delta = \sigma.
\end{equation}
Since $\sigma  = \delta^* - \tau/2$, we find
\begin{equation}
\delta = \delta^* - \frac{\tau/2}{1-s},
\end{equation}
and since $\tau<0$ and $s\in(0,1)$, $\delta>\delta^*$.  Indeed, by raising $\tau$ close to zero,
or lowering $p$ close to $n/2$ (which raises $s$ up to 1), we can obtain
any particular $\delta>\delta^*$. We conclude from
inequalities \eqref{eq:loworder1}, \eqref{eq:loworder2}
and the arithmetic-geometric mean inequality that
\begin{equation}
\left| \int R u^2 \right| \le \epsilon \|\nabla u\|_{W^{1,2}_{\delta^*}}^2 + C \|u\|_{2,\delta}^2.
\end{equation}
This establishes inequality \eqref{eq:RuSquaredEstimate} on a compact manifold, and
we obtain \eqref{eq:RuSquaredEstimate} in the non-compact case by applying
the Poincar\'e inequality \eqref{eq:poincare}.
\end{proof}
\begin{cor}\label{cor:uppersc}
The map
\begin{equation}
u\mapsto \int a|\nabla u|^2 + Ru^2
\end{equation}
is weakly upper semicontinuous on $W^{1,2}_{\delta^*}$.
\end{cor}
\begin{proof}
This follows from the weak upper semicontinuity of $u\mapsto \int|\nabla u|^2$
along with Lemma \ref{lem:RIntegralBound}.
\end{proof}

\begin{defn}
Let $V\subseteq M$ be a measurable set.  The \textit{test functions supported in $V$}
are 
\begin{equation}
A(V) := \left\{ u\in W^{1,2}_{\delta^*}(M): u\not\equiv 0, u|_{V^c}=0\right\}.
\end{equation}
\end{defn}

\begin{defn}
Let $V\subseteq M$ be measurable. The \textit{Yamabe invariant} of $V$ is
\begin{equation}\label{eq:yamabe-inv}
y_g(V) = \inf_{u\in A(V)} Q^y(u).
\end{equation}
If $V$ has measure zero, and hence $A(V)$ is empty, we use the convention
$y_g(V)=\infty$.
\end{defn}
In principle, the infimum in the definition of the Yamabe invariant 
could be $-\infty$.  The following estimate, which will be useful
later in the paper as well, shows that this is not possible. 

\begin{lem}\label{lem:Qplusdelta}
Let $\delta\in\Reals$.
There exist positive constants $C_1$ and $C_2$ such that for all $u\in W^{1,2}_{\delta^*}$,
\begin{equation}\label{eq:basicbound}
\|u\|_{W^{1,2}_{\delta^*}} \le C_1\left[\int a|\nabla u|^2 + R u^2\right] 
+ C_2\|u\|_{2,\delta}^2.
\end{equation}
\end{lem}
\begin{proof}
It is enough to establish inequality \eqref{eq:basicbound} assuming $\delta>\delta^*$.
From Lemma \ref{lem:RIntegralBound}, there is a constant $C$ such that
\begin{equation}
\left|\int R u^2 \right| \le \frac{a}{2} \int |\nabla u|^2 + C \|u\|_{2,\delta}^2
\end{equation}
and hence
\begin{equation}
\int a|\nabla u|^2 + R u^2 \ge \frac{a}{2}\int |\nabla u|^2 - C\|u\|_{2,\delta}^2.
\end{equation}
Consequently
\begin{equation}
\int |\nabla u|^2 \le \frac{2}{a}\left[\int a|\nabla u|^2 + R u^2\right] + \frac{2C}{a} \|u\|_{2,\delta}^2.
\end{equation}
Inequality \eqref{eq:basicbound} now follows trivially in the compact case, 
and follows from the Poincar\'e inequality \eqref{eq:poincare}
in the non-compact case.
\end{proof}

\begin{lem}\label{lem:YamabeBoundedBelow}
For every measurable set $V$, $y(V)>-\infty$.
\end{lem}
\begin{proof}
Let $u_k$ be some minimizing sequence for $Q^y$ normalized
so that $\|u_k\|_{N} = 1$. Lemma \ref{lem:Qplusdelta} and the continuous embedding
$L^N\hookrightarrow L^2_{\delta}$ implies that $u_k$ is uniformly
bounded in $W^{1,2}_{\delta^*}$. Estimate \eqref{eq:RuSquaredEstimate} then
implies that $Q(u_k)$ is uniformly bounded below.
\end{proof}

As one might expect, $y(V)$ is a conformal invariant.
\begin{lem}\label{ConformalInvariance}
Suppose $g' = \phi^{N-2}g$ is a conformally related metric with
$\phi-1\in W^{2,p}_{\tau}$. Then
\begin{equation}
y_{g'}(V) = y_g(V).
\end{equation}
\end{lem}
\begin{proof}
The conformal transformation laws
\begin{equation}
\begin{aligned}
dV_{g'} &= \phi^{N} dV_{g}\\
R_{g'} &= \phi^{1-N}( -a\Delta_{g} \phi + R_g \phi)
\end{aligned}
\end{equation}
together with an integration by parts imply
\begin{equation}
\int_M |\nabla u|_{g'}^2 + R_{g'}u^2\; dV_{g'} = 
\int_M |\nabla (\phi u)|_{g}^2 + R_{g}(\phi u)^2\; dV_{g}
\end{equation}
for all $u\in W^{1,2}_{\delta^*}(M)$.
Since $\| \cdot \|_{g',N} = \|\phi \cdot \|_{g,N}$, it follows that
\begin{equation}\label{eq:ConformalEquality}
Q^y_{g'}(u) = Q^y_{g}(\phi u)
\end{equation}
for all $u\in W^{1,2}_{\delta^*}(M)$ as well.  Since
$A(V)$ is invariant under multiplication by $\phi$, 
$y_{g'}(V) = y_g(V)$.
\end{proof}

We will primarily be interested in the sign of the Yamabe invariant.
\begin{defn}
A measurable set $V\subseteq M$ is called \textit{Yamabe positive}, 
\textit{negative}, or \textit{null} depending on the
sign of $y(V)$.
\end{defn}

The Yamabe invariant involves the critical Sobolev exponent $N$
and hence can be technically difficult to work with.
On a compact manifold, however,
the sign of the Yamabe invariant can be determined from the sign
of the first eigenvalue of the conformal Laplacian. These eigenvalues
enjoy superior analytical properties, and we now describe how
to extend this approach to measurable subsets of compact or asymptotically 
Euclidean manifolds.

For $\delta>\delta^*$ we define the Rayleigh quotients
\begin{equation}
Q_{g,\delta}(u) = \frac{\int a|\nabla u|^2 + R u^2}{\|u\|_{2,\delta}^2}.
\end{equation}
Our previous arguments for the Yamabe quotient
imply that $Q_{g,\delta}$ is well-defined for any $u\in W^{1,2}_{\delta^*}\setminus\{0\}$,
and indeed $Q_{g,\delta}$ is continuous on this set.

\begin{defn} The first \textit{$\delta$-weighted eigenvalue of the conformal Laplacian} is
\begin{equation}\label{eq:EigenvalueConfLaplacian}
\lambda_{g,\delta}(V) = \inf_{u\in A(V)} Q_{g,\delta}(u).
\end{equation}
By convention, if $V$ has measure zero then $\lambda_{g,\delta}(V)=\infty$.  
We will write $Q_\delta$ and $\lambda_\delta$ when the metric is understood.
\end{defn}

The value of $\lambda_\delta(V)$ is not particularly meaningful; it depends
on the choice of weight function $\rho$ and it is not a conformal invariant.
Nevertheless, its sign is a conformal invariant independent of the choice of $\rho$.

\begin{prop}\label{cor:MeasTFAE}
For any measurable set $V\subseteq M$,
the following are equivalent:
\begin{enumerate}
\item $y(V)>0$.
\item $\lambda_\delta(V)>0$ for all $\delta>\delta^*$.
\item $\lambda_\delta(V)>0$ for some $\delta>\delta^*$.
\end{enumerate}
\end{prop}
\begin{proof}  We assume that $V$ has positive measure since
the equivalence is trivial otherwise.
The implication $1 \Rightarrow 2$ follows from the inequality
$\|u\|_{2,\delta} \leq C \|u\|_{N}$ applied to $Q^y$. The implication
$2\Rightarrow 3$ is trivial. So it remains to show that $3 \Rightarrow 1$.

Let $V$ be a measurable set with $\lambda_\delta(V)>0$ for some
 $\delta>\delta^*$. Suppose to produce a contradiction
that $y(V)\leq 0$. Then there is a sequence $u_k \in A(V)$, normalized so that 
$\int a|\nabla u_k|^2 + \|u_k\|_{2,\delta}^2 = 1$, such that $Q^y(u_k) \leq 1/k$. Then 
\begin{align}\label{eq:SignEquivalence1}
\lambda_\delta(V) \|u_k\|_{2,\delta}^2\leq \int a|\nabla u_k|^2 + R u_k^2
      \leq \frac1k \|u_k\|_N^2 \leq \frac{c}k \left[\int a |\nabla u_k|^2 
      +\|u_k\|_{2,\delta}^2\right] \leq \frac{c}{k}
\end{align} by the Sobolev inequality \eqref{eq:sobolev-cpct}.
In particular, $\|u_k\|_{2,\delta}^2 \to 0$. Using inequality 
\eqref{eq:SignEquivalence1}, we also find that
\begin{equation}
\int R u_k^2 \leq \frac{c}k - \int a |\nabla u|^2 \to -1.
\end{equation} However, by Lemma \ref{lem:RIntegralBound}, there exists $C>0$ such that 
\begin{equation}
\left|\int Ru_k^2\right| \leq \frac a2 \|\nabla u_k\|^2_2 + C \|u_k\|_{2,\delta}^2 
      \to \frac12,
\end{equation} which is a contradiction. 
\end{proof}

\begin{cor}\label{cor:SignEquivalence}
For a measurable set $V\subseteq M$, 
the signs of $y(V)$ and $\lambda_\delta(V)$ are the same
for any $\delta>\delta^*$. 
\end{cor}
\begin{proof}
Proposition \ref{cor:MeasTFAE} shows that $y(V)$ is positive if and only if 
$\lambda_\delta(V)$ is also. Choosing an appropriate test function shows that $y(V)$ is
negative if and only if $\lambda_\delta(V)$ is also. Together, these imply that $y(V)$ is
zero if and only if $\lambda_\delta(V)$ is.
\end{proof}
The decay rate $\delta^*$ is critical for Corollary \ref{cor:SignEquivalence}.
For $\delta<\delta^*$, $W^{1,2}_{\delta^*}$ is not contained in 
$L^2_\delta$ and hence our definition of $\lambda_\delta$ 
does not extend to this range.  One could minimize $Q_\delta$ over smooth functions
instead to define $\lambda_\delta$, but using rescaled bump
functions on large balls as test functions it can be shown that 
$\lambda_\delta(\R^n) =0$ for $\delta<\delta^*$ despite the fact that Lemma \ref{lem:poincare}
implies $y(\R^n)>0$.   Note that we have not addressed equality in the threshold
case $\delta=\delta^*$.

We now turn to continuity properties of $\lambda_\delta$.  
Monotonicity is obvious from the definition.
\begin{lem}\label{lem:monotone} Let $\delta>\delta^*$. If $V_1$ and $V_2$ are measurable sets with $V_1\subseteq V_2$,
then $\lambda_\delta(V_1)\ge \lambda_\delta(V_2)$.
\end{lem}
Note that Lemma \ref{lem:monotone} holds even for $V_1=\emptyset$, and that
this relies on our definition $\lambda_\delta(\emptyset)=y(\emptyset)=\infty$.
To obtain more refined properties of $\lambda_\delta$, we start by
showing that minimizers of the Rayleigh quotients
exist and are generalized eigenfunctions.
\begin{prop}\label{prop:eigenfunctions}
Let $V$ be a measurable set with positive measure and let $\delta>\delta^*$.
There exists a non-negative $u \in A(V)$ that minimizes
$Q_\delta$ over $A(V)$.
Moreover, on any open set contained in $V$,
\begin{equation}\label{eq:eigenvalue-eq}
-a\Lap u + R u = \lambda_\delta(V) \rho^{2(\delta^*-\delta)} u.
\end{equation}
\end{prop}
\begin{proof}
Let $u_k$ be a minimizing sequence in $A(V)$; this uses the hypothesis
that $V$ has positive measure.  Without loss of generality we may assume that
each $\|u_k\|_{2,\delta}=1$.  Since
\begin{equation}
a\int_M |\nabla u_k|^2 + R u_k^2 = Q_\delta(u_k),
\end{equation}
and since $u_k$ is a minimizing sequence, Lemma \ref{lem:Qplusdelta}
implies $\{u_k\}$ is bounded in
$W^{1,2}_{\delta^*}(M)$ and hence converges weakly in $W^{1,2}_{\delta^*}(M)$ and
strongly in $L^2_{\delta}(M)$ to a limit $u\in W^{1,2}_{\delta^*}(M)$ with 
$\|u\|_{2,\delta}=1$.  
Since each $u_k=0$ on $V^c$,
from the strong $L^2_{\delta}$ convergence we see $u=0$ on $V^c$,
and since $u\not \equiv 0$ we conclude that $u\in A(V)$.
Weak upper semicontinuity (Corollary \ref{cor:uppersc}) implies
$u$ minimizes $Q_\delta$ over the test functions $A(V)$.  Noting
that $|u|$ is also a minimizer, we may assume $u\ge 0$.

Suppose $V$ contains an open set $\Omega$.  Then any $\phi\in C^\infty_c(\Omega)$
with $\phi\not\equiv 0$ belongs to $A(V)$, and we can differentiate
$Q_\delta(u+t\phi)$ at $t=0$ to find $u$ is a weak solution in $\Omega$ of equation
\eqref{eq:eigenvalue-eq}.
\end{proof}

\begin{lem}[Continuity from above]\label{lem:continuityfromabove}
Let $V\subseteq M$ be a measurable set.
If $\{V_k\}$ is a decreasing sequence of measurable sets with $\cap V_k = V$, then
\begin{equation}\label{eq:CtsFromAbove}
\lim_{k\ra\infty} \lambda_\delta(V_k) = \lambda_\delta(V).
\end{equation}
\end{lem}
\begin{proof} From the elementary monotonicity of $\lambda_\delta$,
$\Lambda = \lim_{k\ra\infty} \lambda_\delta(V_k)$ exists and
\begin{equation}
\lambda_\delta(V_k) \le \Lambda \le \lambda_\delta(V)
\end{equation}
for each $k$. So it is enough to show that
\begin{equation}\label{eq:CtsFromAboveUpper}
\Lambda \ge \lambda_\delta(V).
\end{equation}
We may assume that $\Lambda$ is finite, for inequality \eqref{eq:CtsFromAboveUpper}
is trivial otherwise.  As a consequence, each $V_k$ is nonempty and
Proposition \ref{prop:eigenfunctions} provides minimizers
$u_k$ of $Q_\delta$ over $A(V_k)$ satisfying
$\|u_k\|_{2,\delta}=1$.
For each $k$, since $\|u_k\|_{2,\delta}=1$, 
\begin{equation}\label{eq:CtsFromAboveUpperBound}
\int a|\nabla u_k|^2 + Ru^2_k \le \Lambda.
\end{equation}
From inequality \eqref{eq:CtsFromAboveUpperBound}
and the boundedness of the sequence in $L^{2}_{\delta}(M)$,
Lemma \ref{lem:Qplusdelta} implies the sequence is bounded in $W^{1,2}_{\delta^*}(M)$.
A subsequence converges weakly in $W^{1,2}_{\delta^*}(M)$ and strongly in $L^{2}_{\delta}(M)$
to a limit $v$ with $\|v\|_{2,\delta}=1$. From weak upper semicontinuity 
(Corollary \ref{cor:uppersc}) we conclude $Q_\delta(v)\le \Lambda$ as well.
Moreover, $v\in A(V)$ since $v=0$ on $V_k^c$.  So $\lambda_\delta(v) \le \Lambda$.
\end{proof}
Note that Lemma \ref{lem:continuityfromabove} is false for the Yamabe invariant.
For example, one can take a sequence of balls in $\Reals^n$ that shrink down to
the empty set.  It is easy to see that the Yamabe invariant is scale invariant
and hence is a finite constant along the sequence.  Yet the Yamabe invariant of
the empty set is infinite.  In contrast, if $V_n\searrow \emptyset$,
Lemma \ref{lem:continuityfromabove} implies 
$\lambda_\delta(V_n)\ra\infty$, and in particular
at some point along the sequence $\lambda_\delta(V_n)>0$.  The following
result, which is an extension of \cite{Rauzy95} Lemma 2 to the AE setting, shows 
that in fact $\lambda_\delta(V)$ is positive so long as 
a certain weighted volume is sufficiently small.

\begin{lem}[Small sets are Yamabe positive]
\label{lem:UniformOuterApproximation}
For any $\mu > n$, there exists $C>0$ such that if
$\Vol_\mu(V) := \int_{V} \rho^{-\mu} <C$, $V$ is Yamabe positive.
\end{lem}
\begin{proof}
Suppose that $u\in A(V)$. Define $\delta$ by $(-2\delta -n) \frac{n}2 = -\mu$. 
Note that $\mu>n$ implies that $\delta >\delta^*$. Then, by H\"older's inequality,
\begin{equation}\label{eq:uniform1}
\|u\|_{2,\delta}^2 = \int u^2 \rho^{-2\delta -n} \leq \left(\int u^N \right)^{2/N}
\left( \int_{V} \rho^{(-2\delta-n)\frac{n}{2}}\right)^{2/n} 
= \|u\|_N^2 \Vol_\mu(V)^{2/n}.
\end{equation}
By the Sobolev inequality \eqref{eq:sobolev-cpct}, there exists $C_1$ such that
\begin{equation}\label{eq:uniform2}
\|u\|_N^2 \leq C_1 \left[\int a |\nabla u|^2 + \|u\|_{2,\delta}^2\right].
\end{equation} We also note that Lemma \ref{lem:RIntegralBound}
implies there exists $C_2$ such that
\begin{equation}\label{eq:uniform3}
-C_2 \|u\|_{2,\delta}^2 \leq \frac12 \int a |\nabla u|^2 + \int R u^2.
\end{equation}

Let $\eta$ be defined by $\eta \Vol_\mu(V)^{2/n}C_1 = \frac12$.
Using inequalities \eqref{eq:uniform1}-\eqref{eq:uniform3}, we calculate
\begin{equation}\label{eq:UniformOuter1}
\begin{aligned}
(\eta - C_2) \|u\|_{2,\delta}^2 
&\leq \eta\|u\|_N^2 \Vol_\mu(V)^{2/n} + \int R u^2
 + \frac{1}{2} \int a|\nabla u|^2\\
&\leq \eta\Vol_\mu(V)^{2/n}C_1 \left[\int a |\nabla u|^2 +\|u\|_{2,\delta}^2\right]+ \int R u^2
 + \frac12 \int a|\nabla u|^2\\
&= \int \left(a|\nabla u|^2 + Ru^2\right) + \frac{1}{2}\|u\|_{2,\delta}^2.
\end{aligned}
\end{equation} Dividing through by $\|u\|_{2,\delta}^2$, inequality
\eqref{eq:UniformOuter1} reduces to
\begin{equation}
\eta- C_2 -\frac{1}{2}\leq Q_\delta(u).
\end{equation} As $\Vol_\mu(V) \to 0$, $\eta \to \infty$. Thus there is a $C>0$ such
that if $\Vol_\mu(V) <C$, then $Q_\delta(u)$ has a uniform positive lower bound
for all $u \in A(V)$. Thus $\lambda_\delta(V) >0$, and so $V$ is Yamabe positive
by Corollary \ref{cor:SignEquivalence}.
\end{proof}

In Section \ref{sec:YamabeClassification} below we discuss the relationship between
the Yamabe invariant of an AE manifold and its compactification. 
After compactification, for $\mu=2n$, the condition 
$\Vol_{\mu}(V)<C$ corresponds to the condition that the 
usual volume of the compactified set is sufficiently small. This is exactly
Rauzy's condition, and the other choices of $\mu$ provide a mild generalization
of his result.

\begin{lem}[Strict monotonicity at connected, open sets] \label{InnerApproxNull}
Let $\delta>\delta^*$ and let $\Omega$ be a connected open set.
For any measurable set $E$ in $\Omega$  with positive measure, 
\begin{equation}\label{eq:strictmonotone}
\lambda_\delta(\Omega\setminus E) > \lambda_\delta(\Omega).
\end{equation}
\end{lem}
\begin{proof}
Let $V=\Omega\setminus E$. We may assume $V$ has positive measure, for 
inequality \eqref{eq:strictmonotone} is trivial otherwise. 

Suppose to the contrary that $\lambda_\delta(V) = \lambda_\delta(\Omega)$.
Since $V$ has positive measure, Proposition \ref{prop:eigenfunctions} provides
a function $u\in A(V)$ with $Q_\delta(u) = \lambda_\delta(V)$.
Hence $u$ also is a minimizer of $Q_\delta$ over $A(\Omega)$, and 
Proposition \ref{prop:eigenfunctions} implies that $u$ weakly solves 
\begin{equation}
-a\Delta u + \left[R-\lambda_\delta \rho^{2(\delta^*-\delta)}\right]u =0
\end{equation}
on $\Omega$. Local regularity implies that $u\in W^{2,p}_{\mathrm{loc}}(\Omega)$,
and we may assume after adjusting $u$ on a set of zero measure 
that $u$ is continuous. Since $E$ has positive measure, we can still
conclude that $u$ vanishes at some point in $\Omega$.
Following the argument of Lemma 4 from \cite{Maxwell05b},
we may apply the weak Harnack inequality of \cite{trudinger-measurable}
to conclude that $u$ vanishes everywhere on the connected 
set $\Omega$, and hence on all of $M$.  Since $u\in A(\Omega)$, 
this is a contradiction.
\end{proof}
The connectivity hypothesis in Lemma \ref{InnerApproxNull} is necessary to
obtain strict monotonicity. For example, two disjoint unit balls in $R^n$ 
have the same first eigenvalue as a single unit ball.  On the other hand, 
the assumption that $\Omega$ is open is not optimal, and relaxing this condition
would require a suitable replacement for the weak Harnack inequality.

Although we have not established continuity from below for $\lambda_\delta$,
it holds in certain cases.  The following is a prototypical result
that suffices for our purposes.
\begin{lem}[Continuity from below; prototype] \label{InnerApproxNeg}
Suppose $V$ is measurable.  Let $x_0\in M$ and let $B_r(x_0)$
be the ball of radius $r$ about $x_0$.  Then for any $\delta>\delta^*$
\begin{equation}\label{eq:CtsFromBelow}
\lim_{r\ra 0} \lambda_\delta( V\setminus B_r) = \lambda_\delta( V ).
\end{equation}
\end{lem}
\begin{proof}
Let $u$ be a function in $A(V)$ that minimizes $Q_\delta$.  
Let $\chi_r$ be a radial bump function that equals 0 on $B_{r}(x_0)$, 
equals 1 outside 
$B_{2r}(x_0)$, and has gradient bounded by $2/r$.  Defining $u_r = \chi_r u$ we 
claim that $u_r\ra u$ in $W^{1,2}_{\delta^*}(M)$.  Assuming this for the moment,
we conclude from the continuity of $Q_\delta$ that 
\begin{equation}
\lambda_\delta(V) \le \lambda_\delta(V\setminus B_r) \le Q_\delta( u_r )
\ra Q_\delta(u) = \lambda_\delta(V)
\end{equation}
and hence we obtain equality \eqref{eq:CtsFromBelow}.

To show $u_r\ra u$ in $W^{1,2}_{\delta^*}$, since
$u_r\ra u$ in $L^2_{\delta^*}$, it is enough to show that
$\int |\nabla(u-u_r)|^2\rightarrow 0$.  However,
\begin{equation}\label{eq:ur-to-u}
\int |\nabla(u-u_r)|^2 \le  2 \int (1-\chi_r)^2 |\nabla u|^2 + u^2 | \nabla(1-\chi_r)|^2.
\end{equation}
The first term on the right-hand side of inequality \eqref{eq:ur-to-u} evidently
converges to zero. For the second, we note from H\"older's inequality that
\begin{equation}
\int_{B_{2r}} u^2 \le \left[\int_{B_{2r}} u^N\right]^{\frac{2}{N}}  
\left[\int_{B_{2r}} 1 \right]^{\frac{2}{n}} \le Cr^2 \left[\int_{B_{2r}} u^N\right]^{\frac{2}{N}}.
\end{equation}
Since $u\in L^N_{\mathrm{loc}}$, $\int_{B_{2r}} u^N\to 0$ as $r\ra 0$.  Since 
$\nabla(1-\chi_r)$ is bounded by $c/r$, we conclude that the second term of
the right-hand side of inequality \eqref{eq:ur-to-u} also converges to zero.

\end{proof}

\section{Prescribed Non-Positive Scalar Curvature} \label{sec:PrescribedProblem}

In this section we prove the following necessary and sufficient condition for
being able to conformally transform to non-positive scalar curvature for
AE manifolds with at least one end.

\begin{thm}\label{thm:PrescribedScalarCurvature}
Let $(M^n,g)$ be a $W^{2,p}_\tau$ AE manifold with $p>n/2$ and $\tau \in (2-n,0)$.
Suppose $R'\in L^p_{\tau-2}$ is non-positive.  Then the following are equivalent:
\begin{enumerate}
\item There exists
a positive function $\phi$ with
$\phi-1\in W^{2,p}_\tau$
 and such that the scalar curvature
of $g'=\phi^{N-2}g$ is $R'$.
\item $\{R'=0\}$ is Yamabe positive.
\end{enumerate}
\end{thm}

For compact Yamabe negative manifolds we have the following analogous
result. Since Rauzy's condition
\eqref{eq:rauzycond} is equivalent to the set $\{R'=0\}$ being Yamabe
positive, this is a generalization to lower regularity
and a correction of the proof of part of \cite{Rauzy95} Theorem 1.

\begin{thm}\label{thm:PrescribedCompact}
Let $(M^n,g)$ be a $W^{2,p}$ compact Yamabe negative manifold with $p>n/2$.
Suppose $R'\in L^p$ is non-positive.  Then the following are equivalent:
\begin{enumerate}
\item There exists
a positive function $\phi$ with
$\phi\in W^{2,p}$ and such that the scalar curvature
of $g'=\phi^{N-2}g$ is $R'$.
\item $\{R'=0\}$ is Yamabe positive.
\end{enumerate}
\end{thm}

For the most part, the proof of Theorem \ref{thm:PrescribedCompact} can be obtained
from the proof of Theorem \ref{thm:PrescribedScalarCurvature} 
by treating a compact manifold as an asymptotically Euclidean manifold
with zero ends.  So we
focus on Theorem \ref{thm:PrescribedScalarCurvature} and then 
present the few additional arguments needed for Theorem \ref{thm:PrescribedCompact}
at the end of the section.  

Turning to Theorem \ref{thm:PrescribedScalarCurvature}, 
the  proof that 1) implies 2) is short, so we
delay it and concentrate on the direction 2) implies 1).
Suppose that $\{R'=0\}$ is Yamabe positive.
We will show that we can make the desired conformal change
via a sequence of results proved over the remainder of this section.
It suffices work under the following simplifying hypotheses.
\begin{enumerate}
\item We may assume that the prescribed scalar curvature $R'$ is bounded
since Lemma \ref{lem:LowerScalarCurvature}, which we prove next, shows that we can lower
scalar curvature after first solving the problem for a scalar curvature that is
truncated below.
\item We may assume $\{R'=0\}$ contains a neighborhood of infinity,
since continuity from above (Lemma \ref{lem:continuityfromabove}) 
shows we can truncate $R'$ in a ``small'' neighborhood
of infinity such that its zero set remains Yamabe positive, and
we can subsequently lower scalar curvature after solving the modified problem.
\item We may assume that the initial scalar curvature satisfies $R=0$ in
a neighborhood of infinity, since Lemma \ref{lem:ZeroNearInfinity}, which we prove
below, shows we can initially conformally transform to such a scalar curvature,
and since the hypotheses of Theorem \ref{thm:PrescribedScalarCurvature} are conformally
invariant.
\end{enumerate}

\begin{lem} \label{lem:LowerScalarCurvature}
Suppose $(M, g)$ is a $W^{2,p}_\tau$ AE manifold with $p>n/2$ and $\tau \in (2-n,0)$.
Suppose $R'\in L^{p}_{\tau-2}$. If $R_g \geq R'$, then there exists a positive $\phi$
with $\phi-1 \in W^{2,p}_\tau$ such that $g' = \phi^{N-2} g$ has scalar curvature $R'$.
\end{lem}
\begin{proof}
We seek a solution to $-a\Delta \phi + R\phi = R'\phi^{q-1}$. Note that
$0$ is a subsolution and, since $R\geq R'$, $1$ is a supersolution. By \cite{Maxwell05b}
Proposition 2,
there exists a solution $\phi$ with $0\leq \phi\leq 1$ and $\phi-1\in W^{2,p}_\tau$. 
Since $\phi\ge 0$ solves $-a\Delta\phi+ (R-R'\phi^{q-2})\phi= 0$, and since $\phi\to 1$
at infinity, the weak Harnack inequality \cite{trudinger-measurable} implies $\phi$
is positive.
\end{proof}

\begin{lem}\label{lem:ZeroNearInfinity}
Suppose $(M, g)$ is a $W^{2,p}_\tau$ AE manifold with $p>n/2$ and $\tau \in (2-n,0)$.
There exists $\phi>0$ with $\phi-1 \in W^{2,p}_\tau$
such that the metric $g' =\phi^{N-2}g$ has zero scalar curvature on some neighborhood
of infinity.
\end{lem}
\begin{proof}
We prove the result for a manifold with one end; the extension to
several ends can be done by repeated application of our argument.
Let $E_r$ be the region outside the coordinate ball of radius $r$ in end 
coordinates. By Lemma 
\ref{lem:UniformOuterApproximation}, $y(E_r)>0$ for $r$ large enough.
Following \cite{Maxwell05b} Proposition 3 we claim that 
\begin{equation} \label{eq:DirLap}
-a\Delta + \eta R : \{u \in W^{2,p}_\tau(E_R): u|_{\p E_r} = 0\} \to L^p_{\tau-2}(E_R)
\end{equation} is an isomorphism for all $\eta \in [0,1]$. Because we assume 
homogenous boundary conditions, the argument in \cite{Bartnik86} 
Propositions 1.6 through 1.14 showing that $-a\Delta+\eta R$ is Fredholm of index zero 
requires no changes except imposing the boundary condition. Suppose, then,
to produce a contradiction, that there exists a nontrivial $u$ in the kernel. 
An argument parallel to \cite{Maxwell05b} Lemma 3 implies
$u\in W^{2,p}_{\tau'}$ for any $\tau' \in (2-n,0)$.  In particular,
the extension of $u$ by zero to $M$ belongs to $W^{1,2}_{\delta^*}(M)$ and hence
also to $A(E_r)$. Integration by parts implies $Q^y(u) = 0$, which contradicts
the fact that $E_r$ is Yamabe positive. Thus $-a\Delta+\eta R$ is an isomorphism.

Let $u_\eta$ be the nontrivial solution in $\{u \in W^{2,p}_\tau(E_r): u|_{\p E_r} =0\}$ of
\begin{equation}
  -a\Delta u_\eta + \eta R u_\eta = -\eta R.
\end{equation} Then $\phi_\eta:= u_\eta +1$ solves 
\begin{equation}
-a\Delta \phi_\eta + \eta R \phi_\eta = 0
\end{equation} on $E_r$. Let $I = \{\eta \in [0,1]: \phi_\eta>0\}$. Since
$\phi_0 \equiv 1$, $I$ is nonempty. The set of $u_\eta$ such that $u_\eta>-1$
is open in $W^{2,p}_\tau \subset C^0_\tau$. Thus, by the continuity of the map 
$\eta\mapsto u_\eta$, $I$ is open. Suppose $\eta_0 \in \overline{I}$. If 
$\phi_{\eta_0} =0$ somewhere, the weak Harnack inequality \cite{trudinger-measurable}
implies that $\phi_{\eta_0} \equiv 0$, which
contradicts the fact that $\phi_{\eta_0} \to 1$ at infinity. Thus $\phi_{\eta_0} >0$
on $E_r$, and so $I$ is closed. Thus $I = [0,1]$, and $\phi_1>0$.
Let $\phi$ be an arbitrary positive $W^{2,p}_\tau$ extension of 
$\phi_1|_{E_{r}}$.
\end{proof}

Consider the family of functionals
\begin{equation}\label{eq:DefOfF}
F_q(u) = \int a |\nabla u|^2 + \int R (u+1)^2 - \frac{2}{q} \int R' \left|u+1\right|^q
\end{equation}
for $q\in [2,N)$.

Broadly, the strategy of the proof
is to construct minimizers $u_q$ of the subcritical functionals,
and then establish sufficient control to show that $(1+u_q)$ converges in
the limit $q\ra N$ to the desired conformal factor.
The following uniform coercivity estimate, which we prove following a variation
of techniques found in \cite{Rauzy95}, is the key step in showing the existence
of subcritical minimizers.

\begin{prop}[Coercivity of $F_q$]\label{prop:Coercivity} Suppose
 $\{R'=0\}$ is Yamabe positive, let $\delta>\delta^*$, and let $q_0\in(2,N)$.
For every $B\in\R$  there is a $K>0$ such that for all $q\in [q_0,N)$ and
$u\in W^{1,2}_{\delta^*}$ with $u\ge -1$, if $\|u\|_{2,\delta}>K$ then $F_q(u)>B$.
\end{prop}
\begin{proof}
For $\eta>0$ let
\begin{equation}
A_{\eta} = \left\{ u\in W^{1,2}_{\delta^*}, u\ge -1:
   \int |R'| |u|^2 \le \eta \|u\|_{2,\delta}^2 \int |R'|\right\}.
\end{equation}
Morally, $u\in A_{\eta}$ if it is concentrated on the zero set
\begin{equation}
Z = \{R'=0\},
\end{equation}
with greater concentration as $\eta\ra 0$.

Fix $\calL\in(0,\lambda_\delta(Z))$.
We first claim that there is an $\eta_0<1$
such that if $u\in A_{\eta_0}$, then
\begin{equation}\label{eq:pseudo-yamabe-positive}
\int a |\nabla u|^2 + Ru^2 \ge \calL \|u\|_{2,\delta}^2.
\end{equation}
Suppose to the contrary
that this is false, and let $\eta_k$ be a sequence converging to $0$.
We can then construct a sequence $v_{k}$ with each $v_k\in A_{\eta_k}$ such that
$\|v_{k}\|_{2,\delta}=1$ and
\begin{equation}
\int a |\nabla v_{k}|^2 + Rv_{k}^2< \calL.
\end{equation}
Note that $\calL$ is finite even if $\lambda_\delta(Z)=\infty$.  So
from the boundedness of the sequence $v_k$ in $L^2_\delta$ and 
Lemma \ref{lem:Qplusdelta}, the sequence is bounded in $W^{1,2}_{\delta^*}$,
and a subsequence (which we reduce to) converges weakly in $W^{1,2}_{\delta^*}$
and strongly in $L^2_{\delta}$ to a limit $v$ with $\|v\|_{2,\delta}=1$.
Now
\begin{align}
0 \le \int |R'| v_k^2 &\le \eta_k \int |R'|\rightarrow 0.
\end{align}
Since $|R'| v_k^2 \rightarrow |R'| v^2$ in $L^1$
we conclude $v=0$ outside of $Z$.  From weak upper semicontinuity
(Corollary \ref{cor:uppersc}) we conclude
\begin{equation}\label{eq:v_was_small}
\int a|\nabla v|^2 + Rv^2 \le \calL
\end{equation}
as well. However, since $v$ is supported in $Z$
\begin{equation}
\int a|\nabla v|^2 + Rv^2 \ge \lambda_\delta(Z)\|v\|_{2,\delta}^2 = \lambda_\delta(Z) > \calL,
\end{equation}
which is a contradiction, and establishes inequality \eqref{eq:pseudo-yamabe-positive}.

Let $B\in\R$
and suppose $q\in(q_0,N)$, $u\in W^{1,2}_{\delta^*}$ and $u\ge -1$. We wish to show that
there is a $K$ independent of $q$ so that if $\|u\|_{2,\delta}>K$ then $F_q(u)>B$.
It is enough to find a choice of $K$ under two cases depending on
whether $u\in A_{\eta_0}$ or not. When $u$ is concentrated on $Z$, the coercivity
will follow from the fact that $Z$ is Yamabe positive (as used to obtain
inequality \eqref{eq:pseudo-yamabe-positive}),
and when $u$ is not concentrated on $Z$ the coercivity will follow from the
fact that $R'<0$ away from $Z$.

Suppose that $u\not\in A_{\eta_0}$, so
\begin{equation}\label{eq:u_not_in_A}
\int |R'| |u|^2 > \eta_0 \|u\|_{2,\delta}^2 \int |R'|.
\end{equation}
We calculate
\begin{equation}\label{eq:F_qstep1a}
\begin{aligned}
F_q(u) &= \int a|\nabla u|^2 + \int R(u+1)^2 + \frac{2}{q} \int |R'||u+1|^q\\
&\ge \int a|\nabla u|^2 -2 \int |R|(u^2+1) + \frac{2}{q} \int |R'|(|u|^q-1)\\
&\ge \int \frac{a}{2}|\nabla u|^2 - C\|u\|_{2,\delta}^2 -2 \int |R| + \frac{2}{q} \int |R'| (|u|^q-1)\\
&\ge \int \frac{a}{2}|\nabla u|^2 - C\|u\|_{2,\delta}^2 
      -2\int \left(|R| + \frac{1}{q}|R'|\right) + \frac{2}{q} \int |R'| |u|^q.
\end{aligned}
\end{equation}
Here we have applied Lemma \ref{lem:RIntegralBound} to determine the constant $C>0$, and
have used the fact that $(u+1)^q\ge |u|^q-1$ for $u\ge -1$. Inequality
\eqref{eq:u_not_in_A} and H\"older's inequality imply
\begin{equation}
\begin{aligned}
\eta_0\|u\|_{2,\delta}^2 \int |R'| &<  \int |R'| |u|^2
\le \left(\int |R'| |u|^q\right)^{\frac 2 q }
\left(\int |R'| \right)^{1-\frac 2 q }
\end{aligned}
\end{equation}
and hence
\begin{equation}\label{eq:F_qstep1b}
(\eta_0)^\frac{q}{2}\|u\|_{2,\delta}^q \int |R'| \le \int |R'| |u|^q.
\end{equation}
Using the fact that $\eta_0<1$ and $q<N$, inequalities \eqref{eq:F_qstep1a} and
\eqref{eq:F_qstep1b} imply at last that
\begin{equation}
F_q(u) \geq
\int \frac{a}{2}|\nabla u|^2 - C\|u\|_{2,\delta}^2 -2\int \left(|R| + \frac{1}{q}|R'|\right)
 + \frac{2}{q} (\eta_0)^{\frac{N}{2}} \|u\|_{2,\delta}^q \int |R'|.
\end{equation}
We note that $\int |R'|>0$, for otherwise condition \eqref{eq:u_not_in_A}
is impossible,
and hence the coefficient on $\|u\|_{2,\delta}^q$ is positive. Since $q>2$,
there is a $K$ such that if  $\|u\|_{2,\delta}>K$,
$F_q(u)\ge B$. Note that since $C$ is independent of $q\geq q_0$, so is the choice of $K$.

Now suppose $u\in A_{\eta_0}$, so inequality \eqref{eq:pseudo-yamabe-positive} holds.
Then for any $\epsilon>0$,
\begin{equation}
\begin{aligned}\label{eq:F_q_near_zero}
F_q(u) &\ge \int a|\nabla u|^2 + \int R (u+1)^2\\
&= \int a|\nabla u|^2 + Ru^2 + \int R\left[ (u+1)^2 - u^2\right]\\
&\ge \int a|\nabla u|^2 + Ru^2 - \int |R|\left[ \epsilon u^2 +1+\frac{1}{\epsilon}\right]\\
&\ge (1-\epsilon) \left[\int a|\nabla u|^2 + Ru^2\right] + \epsilon \int (a|\nabla u|^2-2|R|u^2)
- \left(1+\frac{1}{\epsilon}\right) \int |R|\\
&\ge (1-\epsilon) \calL \|u\|_{2,\delta}^2
+ \epsilon\left( \int \frac{a}{2}|\nabla u|^2-C\|u\|_{2,\delta}^2\right)
-\left(1+\frac{1}{\epsilon}\right) \int |R|\\
&\ge \left[ (1-\epsilon) \calL-\epsilon C\right]
\|u\|_{2,\delta}^2 + \epsilon \int \frac{a}{2}|\nabla u|^2
-\left(1+\frac{1}{\epsilon}\right) \int |R|.
\end{aligned}
\end{equation}
Here we have applied Lemma \ref{lem:RIntegralBound} to determine the constant $C$,
inequality \eqref{eq:pseudo-yamabe-positive},
and the fact that $(u+1)^2-u^2 \le \epsilon u^2 +1 + (1/\epsilon)$
for all $u\geq -1$ and all $\epsilon>0$.  We can pick $\epsilon$
sufficiently small such that the coefficient of $\|u\|_{2,\delta}$
in the final expression of inequality \eqref{eq:F_q_near_zero} is at
least $\calL/2$.  Hence there is a $K$ such that
if $\|u\|_{2,\delta}\ge K$, $F_q(u)\ge B$.
Since $C$ is independent of $q\geq q_0$, so is $\epsilon$ and the choice of $K$.
\end{proof}

\begin{lem}\label{lem:ContinuityOfF}
For $q<N$ the operator $F_q$ is weakly upper semicontinuous on
$W^{1,2}_{\delta^*}$.
\end{lem}
\begin{proof}
Lemma \ref{lem:RIntegralBound} together with the weak continuity of continuous linear
maps implies
\begin{equation}
u\mapsto \int a|\nabla u|^2 + R(u+1)^2
\end{equation}
is weakly upper semicontinuous on $W^{1,2}_{\delta^*}$.  Hence it suffices
to show that
\begin{equation}\label{eq:weakmap2}
u\mapsto \int R'|u+1|^{q-1}
\end{equation}
is weakly continuous on $W^{1,2}_{\delta^*}$.  But fixing $\delta>\delta^*$
we know that the embedding $W^{1,2}_{\delta^*}\hookrightarrow L^{q}_{\delta}$
is compact and that the map \eqref{eq:weakmap2} is continuous on $L^q_{\delta}$.
\end{proof}

We now obtain existence of subcritical minimizers from the coercivity of $F_q$,
along with uniform estimates in $W^{1,2}_{\delta^*}$ for the minimizers.

\begin{lem}\label{lem:subcriticalExistence}
For any $q_0 \in (2,N)$, for each $q\in [q_0,N)$, there exists $u_q>-1$,
bounded in $W^{1,2}_{\delta^*}$ independent of $q$, which minimizes $F_q$
and is a weak solution of
\begin{equation}\label{eq:almost-scalar-curvature}
-a\Delta (u_q+1) + R(u_q+1) = R'(u_q+1)^{q-1}.
\end{equation}
Moreover, $u_q \in W^{2,p}_{\sigma}$ for every $\sigma\in(2-n,0)$.
\end{lem}
\begin{proof}
Let $B= \int R + \int |R'|$, let $\delta>\delta^*$, and let $q_0\in (2,N)$. Observe that
\begin{equation}
F_q(0) \le B
\end{equation}
for all $q\in (q_0,N)$. Let $K$ be the constant associated with $B$, $\delta$ and $q_0$
obtained from Proposition \ref{prop:Coercivity}.
Fix $q\in(q_0,N)$ and let $u_k$ be a minimizing sequence in
$W^{1,2}_{\delta^*}$ for $F_q$.
Without loss of generality, we can assume each $u_k\ge -1$ since
$F_q(u_k)=F_q(\max(u_k,-2-u_k))$. We
can assume that each $F_q(u_k)\le F_q(0) \le B$ and hence
Proposition \ref{prop:Coercivity} implies each
$\|u_k\|_{2,\delta}\le K$. Since
\begin{equation}
\int a|\nabla u_k|^2+R(1+u_k)^2 \le F_q(u_k) < B
\end{equation}
as well, Lemma \ref{lem:Qplusdelta} implies that there is a $C>0$ such that each
$\|u_k\|_{W^{1,2}_{\delta^*}}\le C$.
Note that $C$ depends on $K$ and $B$, which are
independent of $q\geq q_0$.
A subsequence (which we reduce to) converges weakly in $W^{1,2}_{\delta^*}$ and
strongly in $L^q_{\delta}$ to a limit $u_q\ge -1$. Lemma \ref{lem:ContinuityOfF}
shows $F_q$ is weakly upper semicontinuous, so $u_q$ is a minimizer.
Moreover, $\|u_q\|_{W^{1,2}_{\delta^*}}\le C$ as well.

Since $u_q$ is a minimizer, we find that $(1+u_q)$ is a weak solution of
\begin{equation}\label{eq:minlinop}
\left[-a\Delta  + R -R'(1+u_q)^{q-2}\right] (1+u_q) = 0.
\end{equation}
Since $R'\in L^\infty_{\text{loc}}$ and since $u_q\in L^{N}_{\text{loc}}$,
an easy computation shows that $R'(1+u_q)^{q-2}\in L^{r}_{\text{loc}}$ for some $r>n/2$.
Since $R\in L^{p}_{\text{loc}}$ and $g\in W^{2,p}_{\text{loc}}$ with $p>n/2$,
we find that the coefficients of the differential operator in brackets
in equation \eqref{eq:minlinop} satisfy the hypotheses of the weak Harnack
inequality of \cite{trudinger-measurable}. Hence, since $1+u_q\ge 0$ and
since the manifold is connected,
either $1+u_q>0$ everywhere or $u_q\equiv-1$. But $u_q$ decays at infinity, and so
we conclude that $1+u_q$ is everywhere positive.

We now bootstrap the regularity of $u_q$, which we know initially belongs to $L^N_{\delta^*}$.  
Fix $\sigma\in(2-n,0)$. Suppose it is known
that for some $r\ge N$ that $u_q\in L^r_\loc$. From equation \eqref{eq:minlinop},
$u_q$ solves
\begin{equation}\label{eq:bootstrap-grr}
-a\Delta u_q = R'(1+u_q)^{q-1}-R(1+u_q).
\end{equation} 
Recall that $R'\in L^\infty_\loc$ and $R\in L^p_\loc$ and both 
have compact support.  Then $R'(1+u_q)^{q-1}$ belongs to $L^{t_1}_\sigma$
with
\begin{equation}
\frac{1}{t_1} = \frac{q-1}{r} \le \frac{1}{r}+\frac{q-2}{N} < \frac{N-1}N
\end{equation}
and $R(1+u_q)$ belongs to  $L^{t_2}_\sigma$ with
\begin{equation}
\frac{1}{t_2} = \frac{1}{r} + \frac{1}{p}.
\end{equation}
Let $t=\min(t_1,t_2)$ and note that $t<p$ since $t_2<p$. From
\cite{Bartnik86} Proposition 1.6 we see that $u_q$ is a strong
solution of \eqref{eq:bootstrap-grr} and from \cite{Bartnik86} Proposition 2.2, 
which implies $\Delta:W^{2,t}_{\sigma}\rightarrow L^t_\sigma$ 
is an isomorphism for $1<t\le p$, we conclude that
$u_q \in W^{2,t}_{\sigma}$. From Sobolev embedding we obtain
$u_q\in L^{r'}_\sigma$ where 
\begin{equation}
\frac 1 {r'} = \frac 1 t-\frac{2}{n},
\end{equation}
so long as $1/t> n/2$, at which point the bootstrap changes as discussed below.
Now
\begin{equation}\label{eq:boothalf1}\begin{aligned}
\frac{1}{t_1}-\frac{2}{n} &\le  \frac{1}{r} + \frac{q-2}{N} - \frac{2}{n} \\
&= \frac{1}{r} +\frac{q}{N} - \left[ \frac{2}{N}+\frac{2}{n}\right]\\
&= \frac{1}{r} +\left[\frac{q}{N}-1\right].
\end{aligned}\end{equation}
Also,
\begin{equation}
\frac{1}{t_2} -\frac{2}{n} = \frac{1}{r} + \left[\frac{1}{p} -\frac{2}n\right].\label{eq:boothalf2}
\end{equation}
Let $\epsilon = \min(1-q/N, 2/n-1/p )$ and note that $\epsilon$
is positive and independent of $r$. Inequalities \eqref{eq:boothalf1}
and \eqref{eq:boothalf2} imply
\begin{equation}
\frac{1}{r'} \le \frac{1}{r} - \epsilon
\end{equation}
Hence, after a finite number of iterations (depending on the size of $\epsilon$, and
hence on how close $q$ is to $N$) we can reduce $1/r$ by multiples of $\epsilon$ 
until $1/r\le \epsilon$.  At this point the bootstrap changes, and in at most
two more iterations we can conclude that $u_q\in L^\infty_\sigma$ and also 
$u_q\in W^{2,p}_\sigma$.

\end{proof}

The uniform $W^{1,2}_{\delta^*}$ bounds of Lemma \ref{lem:subcriticalExistence}
are enough to obtain the existence of a solution $u$ in $W^{2,N/(N-1)}_\sigma$
of equation \eqref{eq:almost-scalar-curvature}
with $q=N$. At the end of Section IV.6 of \cite{Rauzy95} it is
claimed that on a compact manifold in the smooth setting that 
elliptic regularity now implies 
$u$ is smooth.  But in fact this is not quite enough regularity to start a bootstrap:
$W^{2,N/(N-1)}_\sigma$ embeds continuously in $L^N_\sigma$, which is no more
regularity than was known initially.  To start a bootstrap and ensure
the continuity of $u$ we need
the following improved estimate, which follows a modification of the strategy
of \cite{LP87} Proposition 4.4.

\begin{lem}\label{UniformSubcriticalBound}
For each compact set $K$, the minimizers $u_q$ are uniformly bounded in $L^M(K)$
for some $M>N$.
\end{lem}
\begin{proof}
Let $\chi$ be a smooth positive function with compact support that equals 1 in a neighborhood
of $K$.  Let $v = \chi^2 (1+u_q)^{1+2\sigma}$ where $u_q$ is a subcritical minimizer
and where $\sigma$ is a small constant to be chosen later.
Note that since $u_q \in L^{\infty}_{\text{loc}}\cap W^{1,2}_{\text{loc}}$,
$v\in W^{1,2}_{\delta^*}$.  Setting  $w =(1+u_q)^{1+\sigma}$, a short
computation shows that
\begin{equation}\label{eq:subcritbound1}
\int\chi^2 |\nabla w|^2 = -2\frac{1+\sigma}{1+2\sigma} \int \left<\chi\nabla w,w\nabla \chi\right>
+ \frac{(1+\sigma)^2}{1+2\sigma} \int \left<\nabla u_q, \nabla v\right>.
\end{equation}
Applying Young's inequality to the first term on the right-hand side of equation
\eqref{eq:subcritbound1} and merging a resulting piece into the left-hand side we conclude
there is a constant $C_1$ such that
\begin{equation}\label{eq:subcritbound2}
\| \chi \nabla w\|^2_2 \le C_1 \|w\nabla\chi\|_2^2 + 2\frac{(1+\sigma)^2}{1+2\sigma} \int
\left<\nabla u_q, \nabla v\right>.
\end{equation}
Since $u_q$ is a subcritical minimizer,
\begin{equation}\label{eq:subcritbound3}
\begin{aligned}
a\int \left<\nabla u_q, \nabla v\right> &= \int R' (1+u_q)^{q-2} \chi^2 w^2 -
\int R \chi^2 w^2 \\
&\le \left|\int R \chi^2 w^2\right| \\
&\le \epsilon\|\nabla(\chi w)\|_{2}^2  + C_\epsilon \|\chi w\|_{2}^2.
\end{aligned}
\end{equation}
We applied Lemma \ref{lem:RIntegralBound} in the last line and used
the fact that for functions with support contained in a fixed compact set,
weighted and unweighted norms are equivalent.  Note also that obtaining line 2
used the fact that $R'\leq 0$ everywhere.
Noting that there is a constant $C_2$ such that
\begin{equation}\label{eq:subcritbound4}
\| \nabla(\chi w)\|^2_2 \le C_2( \|\chi \nabla w\|_2^2 + \|w\nabla\chi\|_2^2 ),
\end{equation}
we can combine inequalities \eqref{eq:subcritbound2}, \eqref{eq:subcritbound3}, and
\eqref{eq:subcritbound4} to conclude that, upon taking $\epsilon$ sufficiently small
to absorb the term from inequality \eqref{eq:subcritbound3} into the left-hand side,
there is a constant $C_3$ such that
\begin{equation}
\| \nabla(\chi w)\|^2_2 \le C_3 \left[ \|w\nabla\chi\|_2^2 + \|w \chi\|_2^2 \right].
\end{equation}
Finally, from the Sobolev inequality \eqref{eq:sobolev-cpct},
there is a constant $C_4$ such that
\begin{equation}
\| \chi w \|_N^2 \le C_4 \left[ \|w\nabla\chi\|_2^2 + \|w \chi\|_2^2 \right]
\end{equation}
as well. 
Now $u_q$ is bounded uniformly in $L^N$ on the support $K'$ of $\chi$,
and hence we can take $\sigma$ sufficiently small so that $w$ is bounded
independent of $q$ in $L^2(K')$ as well. Thus $(1+u_q)$ is bounded uniformly
in $L^M(K)$ for $M=N(1+\sigma)$.
\end{proof}

\begin{cor} \label{cor:UniformBound} Let $p$ be the exponent such that $g$ is
a $W^{2,p}_{\tau}$ AE manifold and let $\sigma\in (2-n,0)$.
The subcritical minimizers $u_q$ are bounded in $W^{2,p}_{\sigma}$ as $q\rightarrow N$.
\end{cor}
\begin{proof}
Consider a subcritical minimizer $u_q$, which is a weak solution of
\begin{equation}
-a\Lap u_q = -R(1+u_q) +R'(1+u_q)^{q-1}.
\end{equation}
Let $K$ be a compact set containing the support of $R$ and $R'$,
and let $M>N$ be an exponent such that we have uniform bounds on $u_q$
in $L^M(K)$.  We wish to bootstrap this to better regularity for $u_q$.

Since the bootstrap for the two terms is different, we concentrate
first on the interesting term,  $R'(1+u_q)^{q-1}$, and suppose for
the moment that the other term is absent. Let us write
\begin{equation}
\frac 1 M = \frac{1}{N} - \epsilon
\end{equation}
for some $\epsilon>0$.  Now
\begin{equation}
|R'(1+u_q)^{q-1}| \le |R'| (1 + |1+u_q|^{N-1}).
\end{equation}
Since $R'$ is bounded, the term $R'|1+u_q|^{N-1}$
belongs to $L^{s}(K)$ with
\begin{equation}
\begin{aligned}
\frac{1}{s} &= \frac{1}{M}(N-1)\\
&= \left(\frac{1}{N}-\epsilon\right)(N-1) \\
&= \frac{2}{n} + \frac{1}{N} - \epsilon (N-1).
\end{aligned}
\end{equation}
Since $R'$ is zero outside of $K$
we conclude $R'(1+u_q)^{q-1}\in L^{s}_{\sigma}$.
Note that the norm of $R'(1+u_q)^{q-1}$ in $L^{s}_{\sigma}$ 
depends on  the norm of $u_q$ in $L^M(K)$ but is otherwise independent
of $q$.  Since the functions $u_q$ are uniformly bounded in $L^M(K)$,
we obtain control of $R'(1+u_q)^{q-1}$ in $L^{s}_{\sigma}$ independent of $q$.

If $s\le p$ then $s\in (1,p]$ and
we cite \cite{Bartnik86} Proposition 2.2 to conclude 
$u_q\in W^{2,s}_\sigma$
and therefore $u_q\in L^{M'}(K)$ with
\begin{equation}
\frac{1}{M'} = \frac{1}{s}-\frac{2}{n} = \frac{1}{N} - \epsilon (N-1).
\end{equation}
Similarly, after $k$ iterations of this process we would find $u_q$ belongs to
to $W^{2,s}_\sigma$ with
\begin{equation}
\frac{1}{s} = \frac{2}{n} + \frac{1}{N} - \epsilon (N-1)^k
\end{equation}
unless $s > p$, at which point the bootstrap terminates at 
$u_q\in W^{2,p}_{\sigma}$
with norm depending on $\|u_q\|_{L^M(K)}$ (which is independent of $q)$
and the number of iterations needed to reach $s\le p$.  
Note that since $N>2$, we will reach the condition 
$s\ge p$ in a finite number of steps independent of $q$.

Now consider the bootstrap for the term $-R(1+u_q)$ alone.  Write
\begin{equation}
\frac 1 p = \frac 2 n - \epsilon'
\end{equation}
for some $\epsilon'>0$. The term $-R(1+u_q)$ then belongs to $L^{t}(K)$ with
\begin{equation}
\frac{1}{t} = \frac{1}{p}+\frac{1}{M} = \frac{2}{n}-\epsilon' +\frac{1}{M}.
\end{equation}
Note that $1<t<p$ and hence \cite{Bartnik86} Proposition 2.2 implies
$u_q\in W^{2,t}_{\sigma}$.  Note that the norm of $u_q$ in
$W^{2,t}_{\sigma}$ depends on the norm of $u_q$ in $L^M(K)$ but is otherwise
independent of $q$.  Consequently
$u_q$ is controlled in $L^{M'}(K)$ independent of $q$ where
\begin{equation}
\frac{1}{M'} = \frac{1}{t}-\frac{2}{n} = \frac{1}{M} - \epsilon'.
\end{equation}
After $k$ iterations we would find instead
\begin{equation}
\frac{1}{M'} = \frac{1}{M} - k\epsilon'
\end{equation}
and the bootstrap stops in finitely many steps independent of $q$
when $k\epsilon'>1/M$, at which point 
we find that $u_q\in W^{2,p}_{\sigma}$, with norm
independent of $q$.  
There is an exceptional case if $k\epsilon'=1/M$,
but it can be avoided by an initial perturbation of $M$.

The bootstrap in the full case follows from combining these arguments.
\end{proof}

\begin{proof}[Proof of Theorem \ref{thm:PrescribedScalarCurvature}]
\textbf{(2. implies 1.)}\quad The $u_q$ are uniformly bounded in $W^{2,p}_{\sigma}$
by Corollary \ref{cor:UniformBound} for any $\sigma\in(2-n,0)$. 
Thus they converge to some $u$ strongly in
$W^{1,2}_{\delta^*}$ and uniformly on compact sets.
In particular, since the $u_q$ weakly solve \eqref{eq:almost-scalar-curvature}, 
$\phi:= u+1$ weakly solves
\begin{equation}
-a\Delta\phi + R\phi = R'\phi^{N-1}.
\end{equation}

Since each $u_q\geq -1$, $\phi\geq 0$, and since $\phi \to 1$
at infinity, $\phi\not\equiv 0$.  Hence the weak Harnack 
inequality \cite{trudinger-measurable} implies $\phi>0$. 

Since $\sigma\in(2-n,0)$ is arbitrary, $\phi-1\in W^{2,p}_{\tau}$ in particular.
Note that the rapid decay $\sigma\approx 2-n$ uses the fact that $R=0$ near infinity.
The lesser decay rate $\tau$ in the statement of the theorem
stems from the fact that we may have used a conformal factor in $W^{2,p}_\tau$
to initially set $R=0$ near infinity or to lower the scalar curvature after 
changing it to $R'$.

\textbf{(1. implies 2.)}\quad Let $Z=\{R'=0\}$. The case
where $Z$ has zero measure is trivial, for 
then $y(Z) = \infty>0$.  Hence we assume
$Z$ has positive measure and suppose there exists a conformally related metric $g'$ with
scalar curvature $R'$. Let $\delta>\delta^*$ be fixed and let $u$ be a
minimizer of $Q_{g',\delta}$ over $A(Z)$ as provided
by Proposition \ref{prop:eigenfunctions}. Note that
\begin{equation}
\int {R'} u^2 dV_{g'} = 0
\end{equation}
since $R'=0$ on $Z$ and $u=0$ on $Z^c$.  Hence
\begin{equation}
\lambda_{g',\delta}(Z) = Q_{g',\delta}(u) 
= a\frac{\int |\nabla u|^2_{g'}dV_{g'}}{\|u\|_{g',2,\delta}}.
\end{equation}
In particular, $\lambda_{g',\delta}(Z)\ge 0$, and $\lambda_{g',\delta}(Z)=0$
only if $u$ is constant. But $Z$ has positive measure, and therefore
$A(Z)$ does not contain any constants.  Hence $\lambda_{g',\delta}(Z)>0$,
and Proposition \ref{cor:MeasTFAE}  implies $Z$ is Yamabe positive.
\end{proof}

This completes the proof of Theorem \ref{thm:PrescribedScalarCurvature}. 
Turning to the compact case (Theorem \ref{thm:PrescribedCompact})
recall that we started the AE argument with the following inessential simplifying hypotheses:
\begin{enumerate}
\item The prescribed scalar curvature $R'$ is bounded.
\item The prescribed scalar curvature $R'$ has compact support.
\item The initial scalar curvature $R$ has compact support.
\end{enumerate}
The last two of these are trivial when $M$ is compact, and 
the first is justified by  Lemma \ref{lem:LowerScalarCurvatureCompact} below,
which shows that we can lower scalar curvature after first solving the problem 
for a scalar curvature that is truncated below. In the compact case we require
an additional inessential condition which 
will be used in Lemma \ref{lem:UniformLowerBoundForCompact}.
\begin{enumerate}
\setcounter{enumi}{3}
\item We may assume that the initial scalar curvature $R$ is continuous and negative.
Indeed, from Proposition \ref{prop:eigenfunctions} there is a positive function
$\phi$ solving $-a\Delta \phi + R\phi = \lambda_\delta(M) \phi$ on $M$.
Note that $\lambda_\delta(M)<0$ since $g$ is Yamabe negative.  Using
$\phi$ as the conformal factor we obtain a scalar curvature $\lambda_\delta(M) \phi^{2-N}$.
The hypotheses of Theorem 4.2 are conformally invariant and hence unaffected by this change.
\end{enumerate}
\begin{lem} \label{lem:LowerScalarCurvatureCompact}
Suppose $(M, g)$ is a $W^{2,p}$ compact Yamabe negative manifold.
Suppose $R'\in L^{p}$. If $0\geq R \geq R'$, then there exists a positive $\phi$ with
 $\phi \in W^{2,p}$ such that $g' = \phi^{N-2} g$ has scalar curvature $R'$.
\end{lem}
\begin{proof}
We wish to solve
\begin{equation}\label{eq:sclower}
-a \Delta \phi + R\phi = R'\phi^{N-1}.
\end{equation}
Note that $\phi_+=1$ is a supersolution of equation \eqref{eq:sclower}.
To find a subsolution first observe that $R\not \equiv 0$ since the manifold is Yamabe negative.
So, since $-R\ge 0$ and $-R\not\equiv 0$, 
for each $\epsilon>0$ there exists a unique $\phi_\epsilon\in W^{2,p}$ solving
\begin{equation}\label{eq:sceps}
-a\Delta \phi_\epsilon - R\phi_\epsilon  = -R + \epsilon R'.
\end{equation}
When $\epsilon=0$ the solution is $1$, and since $W^{2,p}$ embeds continuously
in $C^0$ we can fix $\epsilon>0$ such that $\phi_{\epsilon}>1/2$
everywhere. We claim that $\phi_-:=\eta \phi_\epsilon$ is a subsolution
if $\eta>0$ is sufficiently small.  Indeed,
\begin{equation}\begin{aligned}
-a\Delta \phi_- + R\phi_- &= \eta\left[ R(2\phi_\epsilon-1) \right] + \eta\epsilon R'\\
& \le \eta\epsilon R'.
\end{aligned}\end{equation}
So $\phi_-$ is a subsolution so long as 
\begin{equation}\label{eq:subsol2}
\eta\epsilon R' \le R'\phi_-^{N-1}
\end{equation}
A quick computation shows that inequality \eqref{eq:subsol2} holds
if $\eta$ is small enough so that $\eta^{2-N}\ge \phi_\epsilon^{N-1}/\epsilon$
everywhere. We can also take $\eta$ small enough so that $\phi_-\le \phi_+=1$, and hence
there exists a solution $\phi\in W^{2,p}$ with $\phi\ge\phi_->0$ 
of equation \eqref{eq:sclower}  (\cite{Maxwell05b} Proposition 2).
\end{proof}

The remainder of the proof of Theorem \ref{thm:PrescribedCompact}
nearly exactly follows the proof of Theorem \ref{thm:PrescribedScalarCurvature}
by treating a
compact manifold as an asymptotically Euclidean manifold with zero ends.
In particular, the cited results of Section 3 apply equally in both cases,
and differences arise only when the following facts are cited.
\begin{itemize}
\item A constant function in $W^{1,2}_{\delta^*}$ is identically zero.
\item The Laplacian is an isomorphism from $W^{2,p}_\sigma$ to $L^p_\sigma$
for $\sigma\in (2-n,0)$.
\end{itemize}
We use the property that constants in $W^{1,2}_{\delta^*}$ vanish just twice:
once in Lemma \ref{lem:subcriticalExistence} in showing $1+u_q\not\equiv 0$,
and once in the final proof of Theorem \ref{thm:PrescribedScalarCurvature}
showing that in the limit $1+u\not\equiv 0$ as well.  The following
lemma provides the alternative argument needed to ensures these
functions do not vanish identically in the compact case.

\begin{lem}\label{lem:UniformLowerBoundForCompact}
Suppose $(M,g)$ is compact and that $R_g$
is continuous and negative. Fix $q_0\in (2,N)$. Then $\|1+u_q\|_2\geq C$ for some $C$
independent of $q\in (q_0,N)$. Moreover, the limit $1+u$ is not identically zero.
\end{lem}
\begin{proof}
Note that for any constant $k$,
\begin{equation}
F_q(k) = (1+k)^2\int R - \frac{2}{q} (1+k)^{q}\int R'.
\end{equation} 
Since $\int R<0$, for any $k\neq -1$ close enough to $-1$, $F_q(k)<0$.
Indeed, there are constants $k_0>-1$ and $c>0$ 
such that $F_q(k_0)<-c$ for all $q\in (q_0,N)$.  But then
\begin{equation}
\int R (1+u_q)^2 \le F_q(u_q) \le F_q(k_0) \le -c
\end{equation}
since $u_q$ minimizes $F_q$. Since $R$ is continuous,
and thus bounded below, $\|1+u_q\|_2\geq C$ for some $C$ independent of $q\in(q_0,N)$.
Since $u_q \to u$ in $L^2$, we also have $\|1+u\|_2 \geq C$, and so $1+u$ is not identically
zero.  
\end{proof}

We use the fact that $\|\Delta u\|_{p,\sigma}$ controls $\|u\|_{W^{2,p}_{\sigma}}$
just twice as well, once in the bootstrap of Lemma \ref{lem:subcriticalExistence}
and once in the bootstrap of Lemma \ref{cor:UniformBound}.  However,
on a compact manifold, $\|u\|_{W^{2,p}}$ is controlled by the sum
of $\|\Delta u\|_{p}$ and $\|u\|_2$, and the coercivity estimate
from Proposition \eqref{prop:Coercivity} ensures that $\|u_q\|_2$
is uniformly bounded as $q\ra N$. This provides the needed extra
control for the bootstraps and 
completes the proof of Theorem \ref{thm:PrescribedCompact}.

\section{Yamabe Classification}\label{sec:YamabeClassification}

In this section we provide two characterizations of the Yamabe class of an
asymptotically Euclidean manifold, one in terms of the prescribed scalar curvature
problem and one in terms of the Yamabe type of the manifold's compactification.
Note that throughout this section AE manifolds have at least one end.

\begin{thm}\label{YamabeClassification}
Suppose $(M,g)$ is a $W^{2,p}_\tau$ AE manifold with $p>n/2$ and $\tau \in (2-n,0)$.
Let $\mathcal{R}_{\le 0}$ be the set of non-positive elements of $L^p_{\tau-2}$.
\begin{enumerate}
\item $M$ is Yamabe positive if and only if 
the set of non-positive scalar curvatures of metrics conformally equivalent
to $g$ is $\mathcal{R}_{\le 0}$.
\item $M$ is Yamabe null if and only if 
the set of non-positive scalar curvatures of metrics conformally equivalent
to $g$ is $\mathcal{R}_{\le 0}\setminus\{0\}$.
\item $M$ is Yamabe negative if and only if 
the set of non-positive scalar curvatures of metrics conformally equivalent
to $g$ is a strict subset of $\mathcal{R}_{\le 0}\setminus\{0\}$.
\end{enumerate}
\end{thm}
\begin{proof}  It suffices to prove the forward implications.

\textbf{1)}\; Suppose $M$ is Yamabe positive, and hence so is every subset.
If $R'\in \mathcal{R}_{\le 0}$, then $\{R'=0\}$ is
Yamabe positive and Theorem \ref{thm:PrescribedScalarCurvature}.
implies $[g]$ includes a metric with scalar curvature $R'$.

\textbf{2)}\; Suppose $M$ is Yamabe null. Since $M$ is open and connected, 
Lemma \ref{InnerApproxNull} implies that if $E\subseteq M$ has
positive measure, then $M\setminus E$ is Yamabe positive.  Hence
for any $R'\in \mathcal{R}_{\le 0}$  with $R'<0$ on a set of positive measure,
$\{R'=0\}$ is Yamabe positive, and Theorem \ref{thm:PrescribedScalarCurvature} 
implies we can conformally transform to a metric with scalar curvature $R'$.
But $R'\equiv 0$ is impossible, for otherwise Theorem \ref{thm:PrescribedScalarCurvature}
would imply $M$ is Yamabe positive.

\textbf{3)}\; Suppose $M$ is Yamabe negative.  Since $M$ is open, Lemma \ref{InnerApproxNeg} 
shows that there is a nonempty open set $W\subseteq M$ such that $M\setminus W$ is 
also Yamabe negative.  Suppose $R'\in L^{p}_{\tau-2}$ is non-positive and supported
in $W$. Then $\{R'=0\}$ contains $M\setminus W$ and is hence Yamabe negative.
But then Theorem \ref{thm:PrescribedScalarCurvature} shows that 
we cannot conformally transform to 
a metric with scalar curvature $R'$. In particular, $R'\equiv 0$ is one of the
unattainable scalar curvatures.
\end{proof}

While Theorem \ref{YamabeClassification} completely the describes
the set of allowable scalar curvatures in cases 1) and 2), 
it does not in case 3).  Of course, we already have demonstrated a necessary and sufficient
criterion for being able to make the conformal change:  
the zero set of $R'$ must be Yamabe positive. Nevertheless, it would be desirable 
to describe this situation more concretely, and there are a
few things that can be said. First, by Lemma \ref{lem:UniformOuterApproximation},
if $R' \in \mathcal{R}_{\le 0}$ and the weighted volume of $\{R'=0\}$ is sufficiently small,
then $\{R'=0\}$ is Yamabe positive, and
thus $g$ is conformally equivalent to a metric with scalar curvature $R'$. 
In particular, if $R'<0$ everywhere, then it is attainable.  Conversely,
by Lemma \ref{InnerApproxNeg},
for any sequence $\{R'_k\} \subset \mathcal{R}_{\le 0}$ such that 
$\{R'_k<0\} \subset B_{1/k}(x_0)$ for some fixed $x_0 \in M$, then for $k$ large enough,
$\{R'_k=0\}$ is Yamabe negative, and thus
$g$ is not conformally equivalent to a metric with scalar curvature $R'_k$.
That is, the strictly negative part of $R'$ cannot be constrained to 
a small ball. Similarly, an argument analogous to the proof of Lemma \ref{InnerApproxNeg} shows that
the complement of a sufficiently ``small'' neighborhood of infinity
is Yamabe negative, and hence the strictly negative part of $R'$ cannot be
constrained to a small neighborhood of infinity.

Our second characterization of the Yamabe class of an AE manifold involves
its compactification.
An AE manifold can be compactified using a conformal factor that decays suitably 
at infinity, and a compact manifold can be transformed into an AE manifold using
a conformal factor with a suitably singularity. We would like to show that
the sign of the Yamabe invariant is preserved under these operations,
and we begin by laying out the details of the compactification/decompactification
procedure.  In particular, there is a precise relationship between the decay of the
metric at infinity and its smoothness at the point of compactification.

\begin{lem}\label{lem:compactification}
Let $p>n/2$ and let $\tau=\frac{n}{p}-2$, so $-2<\tau<0$. 
Suppose $(M,g)$ is a $W^{2,p}_\tau$
AE manifold.
There is a smooth conformal factor $\phi$ that decays to infinity
at the rate $\rho^{2-n}$ 
such that $\bar g= \phi^{N-2} g$ extends to a $W^{2,p}$ metric on
the compactification $\Mbar$ .  

Conversely, suppose
$(\Mbar,\gbar)$ is a compact $W^{2,p}$ manifold, with $p>n/2$ and $p\neq n$.
Given a finite set $\mathcal{P}$
of points in $\Mbar$ there is conformal factor $\phibar$
that is smooth on $M=\Mbar\setminus\mathcal{P}$, has a 
singularity of order $|x|^{2-n}$ at each point of $\mathcal{P}$,
and such that $g=\phibar^{N-2}\gbar$ is a $W^{2,p}_{\tau}$
AE manifold with $\tau = \frac n p -2$.
\end{lem}
\begin{proof} For simplicity we treat the case of only one end.

Let $(M,g)$ be a $W^{2,p}_{\tau}$ AE manifold and let $z^i$ be the Euclidean
end coordinates on $M$, so
\begin{equation}\label{eq:gDecomposition}
g_{ij} = e_{ij} +k_{ij},
\end{equation} with $k\in W^{2,p}_{\tau}$. Let $x^i$ be coordinates given
by the Kelvin transform $x^i = z^i/|z|^{2}$, so $z^i = x^i/|x|^{2}$ as well.

We define a conformal factor $\phi = |z|^{2-n}$ near infinity, and extend 
it to be smooth on the rest of $M$.
Let $\gbar = \phi^{N-2} g$ and let $\Mbar$ be the one-point compactification of $M$,
with $P$ being the point at infinity.
We wish to show that $\gbar$ extends to a $W^{2,p}(\Mbar)$ metric.

Near $P$, $\phi^{N-2} = |z|^{-4}$ and
\begin{align}\label{eq:gbarDecomposition}
\gbar_{ij} &= e_{ij} + \kbar_{ij}
\end{align}
where
\begin{equation}
\kbar_{ij} :=  k_{ij} - \frac{4}{|x|^2} x^a k_{a(i} x_{j)}
 + \frac{4}{|x|^4} x^a x^b k_{ab} x_i x_j = O(k).
\end{equation} 
and $x_a = e_{ab} x^b$. Since $\kbar_{ij}\ra 0$ at $P$,
we set $\gbar_{ij}(P)=e_{ij}$ to obtain a continuous metric, and we 
need to show that $\kbar\in W^{2,p}(\Mbar)$.  
Since $\kbar \in W^{2,p}_\loc(M)$, and since a point is a 
removable set, we need only show that 
the second derivatives of $\kbar$ belong to $L^p(B)$ for 
some coordinate ball $B$ containing $P$.

Let $\pbar$ represent the derivatives in $x^i$ coordinates. Since 
$\frac{\p z}{\p x} = O(|x|^{-2})$, we calculate
\begin{equation} \label{eq:ChainRule} \begin{aligned}
\pbar \kbar &= O(\p k)O(|z|^2) + O(k)O(|z|)\\
\pbar^2 \kbar &= O(\p^2k) O(|z|^4) + O(\p k) O(|z|^3) + O(k)O(|z|^{2}).
\end{aligned}\end{equation}
In order to show $\pbar^2 \kbar \in L^p(B)$, it is sufficient to show that each
of the three terms in equation \eqref{eq:ChainRule} is in $L^p(B)$.

Note that near infinity
\begin{equation}
d\Vbar  = \phi^N dV = |z|^{-2n} dV.
\end{equation} Hence the $L^p$ norm of the
$O(k)O(|z|^2)$ term of equation \eqref{eq:ChainRule} is
controlled by
\begin{equation}\begin{aligned}
\int \left(O(k)O(|z|^2)\right)^p |z|^{-2n}dV 
&= \int O\left(|k|^p\right) O\left(|z|^{2p-2n}\right) dV\\
&\leq C\|k\|_{W^{2,p}_{\tau}}^p,
\end{aligned}\end{equation}
where we have used the equality
\begin{equation}
2p-2n = -n- \tau p
\end{equation}
and equation \eqref{eq:weighted-norm} defining the weighted norm.
Hence the $O(k)O(|z|^2)$ term of equation \eqref{eq:ChainRule}
belongs to $L^p(B)$. The two remaining terms have the same asymptotics and 
similar calculations show that they belong to $L^p(B)$ as well.

For the converse, consider a $W^{2,p}$ compact manifold $(\Mbar, \gbar)$
with $p>n/2$ and $p\neq n$. Let $P$ be a point to remove to obtain 
$M=\Mbar\setminus\{P\}$.  Since $\gbar$ is continuous we can find smooth 
coordinates $x^i$ near $P$ such that $\gbar = e+\kbar$ for some $\kbar \in W^{2,p}$
which vanishes at $P$.
Moreover, if $p>n$ then $\gbar$ has H\"older continuous derivatives and the proof
of \cite{Aubin98} Proposition 1.25 shows we can additionally assume
these are normal coordinates (i.e., the first derivatives of $\kbar$ vanish at $P$).
Finally, since the result we seek only involves properties of $\kbar$ local to $P$, we can 
assume that $\kbar=0$ except in a small coordinate ball $B$ near $P$.

We claim there is a constant $C$ such that
\begin{align}
\label{eq:FiniteKbar}
\int_{B} \frac{|\kbar|^p}{|x|^{2p}} &\le C\int_{B} |\pbar^2 \kbar|^p d\Vbar\qquad\text{and}\\
\label{eq:FiniteKbar2}
\int_{B}  \frac{|\pbar \kbar|^p}{|x|^p}d\Vbar &\leq C \int_{B} |\pbar^2 \kbar|^p d\Vbar.
\end{align}
Assuming for the moment that this claim is true,
let $z_i = x_i/|x|^2$. Let $\phibar = |x|^{2-n}$ near $P$ and extend 
$\phibar$ as a positive smooth function on the remainder of $M$. 
Let $g= \phibar^{N-2} \gbar$.
Near $P$, $\phibar^{N-2} = |x|^{-4}$ and so $g= e+k$ near infinity, where
\begin{equation}
k_{ij} := \kbar_{ij} - \frac{4}{|z|^2} z^a \kbar_{a(i} z_{j)} 
+ \frac{4}{|z|^4} z^a z^b \kbar_{ab} z_i z_j = O(\kbar).
\end{equation} 
Since $k\in W^{2,p}_{\loc}$, we need only establish the desired
asymptotics at infinity. 

A computation similar to the one leading to equation \eqref{eq:ChainRule} shows
\begin{equation}\begin{aligned}
\p k &= O(\pbar \kbar)O(|x|^2) + O(\kbar)O(|x|)\\
\p^2 k &= O(\pbar^2 \kbar) O(|x|^4) + O(\pbar \kbar) O(|x|^3) + O(\kbar)O(|x|^{2}).
\end{aligned}\end{equation}
Also, $d\Vbar = |z|^{-2n} dV$ near $P$. Hence
\begin{align}
\int |\p^2 k|^p |z|^{4p-2n} dV &= \int |\p^2 k|^p |x|^{-4p}|x|^{2n} dV \\
&=\int \left(O(\pbar^2 \kbar)\right)^p +\left(O(\pbar\kbar)O(|x|^{-1})\right)^p
+ \left(O(\kbar)O(|x|^{-2})\right)^p d\Vbar.\label{eq:Decompactification1}
\end{align} From inequalities \eqref{eq:FiniteKbar} and \eqref{eq:FiniteKbar2},
quantity \eqref{eq:Decompactification1} is finite.  Noting
 \begin{equation}
4p-2n = -n -\tau p + 2p
 \end{equation}
we conclude $|\p^2 k|\in L^{p}_{\tau-2}$, as desired. A similar calculation shows that
$|\p k| \in L^{p}_{\tau-1}$ and $|k| \in L^p_{\tau}$.  This concludes the proof,
up to establishing inequalities \eqref{eq:FiniteKbar} and \eqref{eq:FiniteKbar2}.

Theorem 1.3 of \cite{Bartnik86}
implies that
\begin{align}\label{eq:HardysInequality}
\int_{B} \frac{|f|^p}{|x|^{2p}} d\Vbar \leq c \int_{B} \frac{|\pbar f|^p}{|x|^p}d\Vbar
        \leq C \int_{B} |\pbar^2 f|^p d\Vbar <\infty
\end{align} for smooth functions $f$ that are compactly supported in $B$
and vanish in a neighborhood of $P$.
This inequality relies on the fact that $p\neq n$, which corresponds to the
condition $\delta=0$ in \cite{Bartnik86} Theorem 1.3.

Let $f_n$ be a sequence of smooth functions vanishing near $P$ that
converges to $\kbar$ in $W^{2,p}$; such a sequence exists since $\kbar=0$
at $P$, since $\partial \kbar =0$ at $P$ if $p>n$, and since we have assumed
that $\kbar$ vanishes outside of $B$.
By reduction to a subsequence we may assume that the values and first derivatives
of sequence converge pointwise a.e., and using Fatou's Lemma we find
\begin{equation}
\begin{aligned}
\int_{B} \frac{|\kbar|^p}{|x|^{2p}} &\leq \liminf_{n\to \infty} \int_{B} \frac{|f_n|^p}{|x|^{2p}}\\
      &\leq C \lim_{n\to\infty} \int_{B} |\pbar^2 f_n|^p d\Vbar \\
      &= C\int_{B} |\pbar^2 \kbar|^p d\Vbar <\infty.
\end{aligned}
\end{equation}
This is inequality \eqref{eq:FiniteKbar}, and a similar argument shows that
inequality \eqref{eq:FiniteKbar2} holds as well.
\end{proof}

The threshold $\tau=-2$ in Lemma \ref{lem:compactification}
arises because there is a connection between the rate
of decay of the AE metric and the rate of convergence of the metric at the
point of compactification in a chosen coordinate system: roughly speaking,
decay of order $\rho^{\tau}$
corresponds to convergence at a rate of $r^{-\tau}$.  For a generic smooth metric we can
use normal coordinates to obtain convergence at a rate of $r^2$, but we cannot
expect to do better generally.  Hence the decompactification of a smooth metric will
typically not decay faster than $\rho^{-2}$.  Looking at the
proof of Lemma \ref{lem:compactification}, we note that it can be 
readily extended to $k>2$ to show that
a $W^{k,p}_\tau$ AE metric with $k\geq 2$, $p>n/k$
and $\tau=(n/p)-k$
can be compactified to a $W^{k,p}$ metric. But the decay condition 
$\tau=(n/p)-k$ is quite restrictive for $k>2$:
smooth metrics decompactify generally to metrics with decay $O(\rho^{-2})$,
but compactification of a $W^{k,p}_{-2}$ metric would not be known to be
$C^3$, regardless of how high $k$ and $p$ are.  A more refined analysis
for $k>2$ would need to take into account asymptotics of the Weyl or Cotton-York
tensor, and we point to Herzlich \cite{Herzlich97} for related results
in the $C^k$ setting.

\begin{prop}\label{prop:yAE=yCpct}
Let $(M,g)$ and $(\Mbar,\gbar)$ be a pair of manifolds as in Lemma
\ref{lem:compactification}, related by $g = \phibar^{N-2} \gbar$.
Then $y_g(M)=y_{\gbar}(\Mbar)$.
\end{prop}
\begin{proof}
For simplicity we assume that $M$ has one end.  Let $P\in \Mbar$ 
be the singular point of $\phibar$. 
Note that $W^{1,2}_c(M)$ is dense in $W^{1,2}_{\delta^*}(M)$
and that
\begin{equation}
S_P:=W^{1,2}(\Mbar)\cap\{u: u|_{B_r(P)}=0 \textrm{ for some } r>0\}
\end{equation} is dense in
$W^{1,2}(\Mbar)$ since $2<n$. 
From upper semicontinuity of the Yamabe quotient, the Yamabe
invariants of $g$ and $\gbar$ can be computed by minimizing
the Yamabe quotient over $W^{1,2}_c$ and $S_P$ respectively.
Note that $u \mapsto \phibar u$ is a bijection between $W^{1,2}_c(M)$
and $S_p$. The proof of Lemma \ref{ConformalInvariance}
shows that for $u\in W^{1,2}_c$,
\begin{equation}
Q_g^y(u) = Q_{\gbar}^y(\phibar u)
\end{equation} 
and hence $y_g(M) = y_{\gbar}(\Mbar)$.
\end{proof}

Combining Lemma \ref{lem:compactification} and Proposition \ref{prop:yAE=yCpct}
we obtain our second classification.
\begin{prop}\label{prop:classifybycompactify}
Let $(M,g)$ be a $W^{2,p}_{\tau}$ AE manifold with $\tau \le \frac{n}{p}-2$.
Then $(M,g)$ is Yamabe positive/negative/null if and only if
some conformal compactification, as described in Lemma \ref{lem:compactification},
has the same Yamabe type.
\end{prop}

Consequently, Yamabe classification on AE manifolds has the same topological flavor
as in the compact setting. For instance, since the
torus does not allow a Yamabe positive metric, the decompactified torus,
which is diffeomorphic to $\R^n$ with a handle, does not allow a metric with nonnegative
scalar curvature.

We mention an application of Proposition \ref{prop:classifybycompactify}
to general relativity. In general relativity, spacetimes can be
constructed by specifying initial data in the form of a Riemannian manifold $(M,g)$ and a
symmetric (0,2)-tensor $K$, and then solving a
hyperbolic evolution problem to construct an ambient Lorentzian spacetime
such that $g$ and $K$ are the induced metric and second fundamental form
of the initial hypersurface.
However, the initial data cannot be freely specified; it
must satisfy the Einstein constraint equations,
\begin{equation} \label{eq:Constraints}
\begin{aligned}
R - |K|^2 + \tr K^2 = r,\\
\di K - d \, \tr K = j,
\end{aligned}
\end{equation} where $r$ is the energy density and $j$ is the momentum density
of matter. 
It is natural to suppose that the energy density $r$ is everywhere nonnegative,
which is known as the weak energy condition.
If the initial data is maximal, i.e., if the mean curvature $\tr\, K$ is zero, then
the weak energy condition implies $R\geq 0$. Thus, if 
the compactification of an AE manifold has a topology that does not
admit a Yamabe positive metric, then
the original AE manifold does not allow maximal initial data satisfying
the weak energy condition. 

\section{Acknowledgements}

We would like to thank James Isenberg and Rafe Mazzeo
for useful discussions. This research
was partially supported by the NSF grant DMS-1263431. This material is based upon work
supported by the National Science Foundation under Grant No. 0932078 000 while the authors
were in residence at the Mathematical Sciences Research Institute in Berkeley, California,
during the fall of 2013.

\bibliographystyle{alpha}
\bibliography{KnownResults,AEConformal,MaxwellAdditions}

\begin{thebibliography}{Max05b}

\bibitem[Aub76]{Aubin76}
T.~Aubin.
\newblock {\'E}quations diff{\'e}rentielles non lin{\'e}aires et probl{\'e}me
  de {Y}amabe concernant la courbure scalaire.
\newblock {\em J. Math. Pures Appl.}, 55:269--296, 1976.

\bibitem[Aub98]{Aubin98}
T.~Aubin.
\newblock {\em Some Nonlinear Problems in Riemannian Geometry}.
\newblock Springer, 1998.

\bibitem[Bar86]{Bartnik86}
R.~Bartnik.
\newblock The mass of an asymptotically flat manifold.
\newblock {\em Comm. Pure Appl. Math.}, 39(5):661--693, 1986.

\bibitem[BE87]{Bourguignon:1987ge}
J.~Bourguignon and J.~Ezin.
\newblock {Scalar Curvature Functions in a Conformal Class of Metrics and
  Conformal Transformations}.
\newblock {\em Transactions of the American Mathematical Society},
  301(2):723--736, June 1987.

\bibitem[CB81]{BC81}
M.~Cantor and D.~Brill.
\newblock The {L}aplacian on asymptotically flat manifolds and the
  specification of scalar curvature.
\newblock {\em Compositio Math.}, 43(3):317--330, 1981.

\bibitem[DGI15]{DGI13}
J.~Dilts, R.~Gicquaud, and J.~Isenberg.
\newblock A nonexistence result for the conformal constraint equations on
  asymptotically {E}uclidean manifolds.
\newblock {\em Preprint}, 2015.

\bibitem[ES86]{ES86}
J.~Escobar and R.~Schoen.
\newblock Conformal metrics with prescribed scalar curvature.
\newblock {\em Invent. Math.}, 86(2):243--254, 1986.

\bibitem[FCS80]{FCS80}
D.~Fischer-Colbrie and R.~Schoen.
\newblock The structure of complete stable minimal surfaces in {$3$}-manifolds
  of nonnegative scalar curvature.
\newblock {\em Comm. Pure Appl. Math.}, 33(2):199--211, 1980.

\bibitem[Fri11]{Friedrich11}
H.~Friedrich.
\newblock Yamabe numbers and the {B}rill-{C}antor criterion.
\newblock {\em Ann. Henri Poincar\'e}, 12(5):1019--1025, 2011.

\bibitem[Her97]{Herzlich97}
M.~Herzlich.
\newblock Compactification conforme des vari\'et\'es asymptotiquement plates.
\newblock {\em Bull. Soc. Math. France}, 125(1):55--91, 1997.

\bibitem[LP87]{LP87}
J.~Lee and T.~Parker.
\newblock The {Y}amabe problem.
\newblock {\em Bull. Amer. Math. Soc. (N.S.)}, 17(1):37--91, 1987.

\bibitem[Max05a]{Maxwell05}
D.~Maxwell.
\newblock Rough solutions of the {E}instein constraint equations on compact
  manifolds.
\newblock {\em J. Hyperbolic Differ. Equ.}, 2(2):521--546, 2005.

\bibitem[Max05b]{Maxwell05b}
D.~Maxwell.
\newblock Solutions of the {E}instein constraint equations with apparent
  horizon boundaries.
\newblock {\em Comm. Math. Phys.}, 253(3):561--583, 2005.

\bibitem[Rau95]{Rauzy95}
A.~Rauzy.
\newblock Courbures scalaires des vari\'et\'es d'invariant conforme n\'egatif.
\newblock {\em Trans. Amer. Math. Soc.}, 347(12):4729--4745, 1995.

\bibitem[Sch84]{Schoen84}
R.~Schoen.
\newblock Conformal deformation of a {R}iemannian metric to constant scalar
  curvature.
\newblock {\em J. Diff. Geom.}, 20:479--495, 1984.

\bibitem[Tri76a]{triebel-weighted-one}
H.~Triebel.
\newblock Spaces of {K}udrajavcev type. {I}. {I}nterpolation, embedding, and
  structure.
\newblock {\em J. Math. Anal. Appl.}, 56(2):253--277, 1976.

\bibitem[Tri76b]{triebel-weighted-two}
H.~Triebel.
\newblock Spaces of {K}udrajavcev type. {II}. {S}paces of distributions:
  duality, interpolation.
\newblock {\em J. Math. Anal. Appl.}, 56(2):278--287, 1976.

\bibitem[Tru68]{Trudinger68}
N~. Trudinger.
\newblock Remarks concerning the conformal deformation of {R}iemannian
  structures on compact manifolds.
\newblock {\em Ann. Scuola Norm. Sup. Pisa}, 22:265--274, 1968.

\bibitem[Tru73]{trudinger-measurable}
N.~Trudinger.
\newblock Linear elliptic operators with measurable coefficients.
\newblock {\em Ann. Scuola Norm. Sup. Pisa (3)}, 27:265--308, 1973.

\bibitem[Yam60]{Yamabe60}
H.~Yamabe.
\newblock On a deformation of {R}iemannian structures on compact manifolds.
\newblock {\em Osaka Math. J.}, 12:21--37, 1960.

\end{thebibliography}
\end{document}